\documentclass[reqno]{amsart}

\usepackage{amsmath,amssymb,amsthm,amsfonts,enumerate, enumitem,hyperref,cleveref}
\usepackage{dsfont}
\usepackage[all]{xy}
\usepackage[T1]{fontenc}
\usepackage{tikz}
\usepackage{pgfplots}
\pgfplotsset{compat=1.15}
\usepackage{mathrsfs}
\usetikzlibrary{arrows}

\DeclareMathOperator{\Span}{span}

\DeclareMathOperator{\rank}{rank}

\theoremstyle{plain}

\newtheorem{thrm}{Theorem}[section]
\newtheorem{cor}[thrm]{Corollary}
\newtheorem{prop}[thrm]{Proposition}
\newtheorem{lem}[thrm]{Lemma}

\theoremstyle{definition}
\newtheorem{defn}[thrm]{Definition}
\newtheorem{rem}[thrm]{Remark}
\newtheorem{exm}[thrm]{Example}

\crefrangeformat{equation}{#3(#1)#4--#5(#2)#6}

\crefname{thrm}{Theorem}{Theorems}
\crefname{theorem}{Theorem}{Theorems}
\crefname{lem}{Lemma}{Lemmas}
\crefname{cor}{Corollary}{Corollaries}
\crefname{prop}{Proposition}{Propositions}
\crefname{defn}{Definition}{Definitions}
\crefname{exm}{Example}{Examples}
\crefname{rem}{Remark}{Remarks}
\crefname{conj}{Conjecture}{Conjectures}
\crefname{quest}{Question}{Questions}
\crefname{section}{Section}{Sections}
\crefname{equation}{\unskip}{\unskip}
\crefname{enumi}{\unskip}{\unskip}
\crefname{subsection}{Subsection}{Subsections}

\newcommand{\af}{\alpha}

\newcommand{\lb}{\lambda}

\newcommand{\vf}{\varphi}

\newcommand{\CC}{\mathbb{C}}

\newcommand{\sst}{\subseteq}

\newcommand{\impl}{\Rightarrow}

\begin{document}
	\title[Bilinear maps that have product property at a compact element]{Bilinear maps on C$^*$-algebras that have product property at a compact element}	
	
	\author{Jorge J. Garc{\' e}s}
	\address{Departamento de Matem{\' a}tica Aplicada a la Ingenier{\' i}a Industrial, ETSIDI, Universidad Polit{\' e}cnica de Madrid, Madrid, Spain}
	\email{j.garces@upm.es}
	
	\author{Mykola Khrypchenko}
	\address{Departamento de Matem\'atica, Universidade Federal de Santa Catarina,  Campus Reitor Jo\~ao David Ferreira Lima, Florian\'opolis, SC, CEP: 88040--900, Brazil \and CMUP, Departamento de Matemática, Faculdade de Ciências, Universidade do Porto, Rua do Campo Alegre s/n, 4169--007 Porto, Portugal}
	\email{nskhripchenko@gmail.com}

	
	\begin{abstract}
 We study bounded bilinear maps on a C$^*$-algebra $A$ having product property at $c\in A$. This leads us to the question of when a C$^*$-algebra is determined by products at $c.$
  In the first part of our paper, we investigate this question for compact C$^*$-algebras, and in the second part, we deal with von Neumann algebras having non-trivial atomic part. Our results are applicable to descriptions of homomorphism-like and derivation-like maps at a fixed point on such algebras.
	\end{abstract}
	
	\maketitle
	
	\tableofcontents
	
	\section*{Introduction}

 One of the popular questions in the theory of linear preservers is: to what extent are the properties of a linear map determined by its values on some specific elements? For instance, restricting the definition of a homomorphism or a derivation of an algebra $A$ to pairs of elements whose product belongs to a fixed subset of $A$, we get a linear map which might no longer be a homomorphism or a derivation. Probably, the most studied class of such maps is that of \textit{zero-product preservers}~\cite{Chebotar-Ke-Lee-Wong03}, i.e. linear maps $\vf:A\to A$ such that $\vf(a)\vf(b)=0$, whenever $ab=0$. The derivation-like analog satisfies $\vf(ab)=a\vf(b)+\vf(a)b$, if $ab=0$, and is sometimes called a linear map \textit{derivable at zero} (see, for example, the introduction of~\cite{AbAdJamPeralta-AntiDer-2020}). Both of the notions generalize to bilinear maps $B:A\times A\to X$ \textit{preserving zero products}~\cite{Alam_Extr_Villean_HomDer_2007}, i.e. satisfying $ab=0\impl B(a,b)=0$. A typical example is a map of the form $B(a,b)=T(ab)$, where $T$ is linear. If there are no other zero products preserving bilinear maps $B$, then the algebra $A$ is said to be \textit{zero product determined} (or \textit{zpd}). For example, the matrix algebra $M_n(R)$ over a commutative unital ring $R$ is zpd (more generally, any unital algebra generated by idempotents is zpd~\cite[Theorem 2.3]{Bresar-zpd}). An appropriate notion of a \textit{zpd Banach algebra} has been introduced in~\cite{Alam_Bre_Ex_Vill_ZP}. It is worth noticing that every C$^*$-algebra is a zpd Banach algebra (see~\cite[Section 5.3]{Bresar-zpd}).
 
 A natural step towards generalizing the above notions is to consider products $ab$ equal to a fixed, not necessarily zero, element $c\in A$. Chebotar, Ke, Lee and Shiao~\cite{Chebotar-Ke-Lee-Shiao05} proved that a bijective additive map on a division ring preserving products equal to the identity element is either an automorphism multiplied by a central element or an anti-automorphism multiplied by a central element. Catalano extended the above description, permitting the products to be an arbitrary invertible element of a division ring $D$~\cite{Catalano18} or even of the matrix ring $M_n(D)$ over $D$~\cite{Catalano-Hsu-Kapalko19}. The case where $ab$ is a rank-one idempotent of $M_n(\mathbb{C})$ was treated by Catalano in~\cite{Catalano21}, followed by the case where $ab$ is a nilpotent element of $M_n(\mathbb{C})$~\cite{Catalano-Chang-Lee21}. Perhaps, the most interesting situation for our purposes was considered in~\cite{Catalano-Julius21b}, where $ab$ is a matrix from $M_n(\mathbb{C})$ of rank at most $n-2$. The corresponding product preservers turn out to be automorphisms multiplied by scalars. In the case of infinite-dimensional Banach algebras, we can mention \cite{Burgos-Cabello-Peralta19} where the authors study linear mappings between C$^*$-algebras that are $^*$-homomorphisms at products $ab$ belonging to a fixed subset which can be  $\{1 \}$, $\{u\}$ or $\{p,1-p\}$ with $u$ being a unitary element and $p$ a projection. More recently, Julius \cite{Julius-JMAA-22} has studied linear maps from $B(X)$, where $X$ is an infinite-dimensional Banach space, to an associative algebra $A$ \textit{that preserve products at a fixed element}. The key result for our purposes is \cite[Proposition 2.1]{Julius-JMAA-22} where it is shown that every linear map $\vf:B(X)\to A$ that preserves products at a finite-rank operator preserves zero products.
 
%\red{cite all derivable points ????}

Our main object of study in this paper is bounded bilinear maps $V$ on a C$^*$-algebra $A$ that satisfy $V(a,b)=\vf(ab)$ for some bounded linear map $\vf$ on $A$ \textit{and for all $a,b\in A$ with $ab$ being equal to a fixed element $c\in A$}. We call them maps having \textit{product property at $c$}. Our approach consists in proving that, under certain assumptions on $A$ and $c$, any map $V$ with product property at $c$ preserves zero products. Since $A$ is zpd, the latter implies that there exists a bounded linear map $\psi$ on $A$ (a priori, different from $\vf$) such that $V(a,b)=\psi(ab)$ \textit{for all $a,b\in A$}. Similarly to the zpd-case, in such a situation we say that $A$ is \textit{determined by products at $c$}.

In the case $A=M_n(\CC)$, the assumption $\rank(c)\le n-2$ coming from \cite{Catalano-Julius21b} cannot be dropped as shown in \cref{exm-rank(c)=n-1,exm-rank(c)=n}. This forces us to introduce what we call the \textit{rank-hypothesis} on $c$. It has two versions \cref{RH,VRH} working in the contexts of compact C$^*$-algebras and von Neumann algebras, respectively. 

In \cref{sec-compact} we prove the first main result of our paper that any compact C$^*$-algebra $A$ is determined by products at any element $c$ satisfying the rank-hypothesis (\cref{thm main compact c*}). The proof is based on two important technical \cref{lem S zp-dense in compact C*,lem compact star factor thu zp}, one of them saying that the set $MinPI(A)$ of minimal partial isometries of $A$ is \textit{zp-dense} in $A$, and another one showing that $c$ (under the rank-hypothesis) \textit{factorizes through zero products} in $MinPI(A)$.

In \cref{sec-von-Neumann} we deal with von Neumann algebras $A$ having non-trivial atomic part. Such algebras are proved to be determined by products at any finite-rank element $c\in A$ satisfying the rank-hypothesis (see \cref{cor von neumann finit bounded}). On the other hand, if $c$ is a compact element of $A$ (satisfying the rank-hypothesis), then $A$ is determined by product at $c$ in a slightly different sense involving weak$^*$-topology (\cref{thhm von Neumann algebras}). 

Finally, in \cref{sec-appl} we consider more specific cases of bilinear maps, namely, those coming from pairs of linear maps with product property at $c$ (in particular, from a single linear map with product property at $c$) and from linear maps derivable at $c$. We first characterize in \cref{prop zp pairs bij finite case} pairs of linear maps between C$^*$-algebras $A$ and $B$ that have product property at zero. Then, we combine in \cref{cor pairs with pp at c} this result with the results from \cref{sec-compact,sec-von-Neumann} to obtain descriptions of pairs of linear maps having product property at $c$ under the corresponding assumptions on the algebras $A$ and $B$ and on $c$. Such linear maps turn out to be \textit{weighted homomorphisms}. In the case of bijective or symmetric linear maps having product property at $c$, similar descriptions have been obtained in \cref{thm map with pp at finite rank element,thm map-pp-at-c-compact case} under weaker assumptions. In \cref{cor derivation at c} we treat several cases, where a linear map derivable at $c$ from a C$^*$-algebra $A$ to a Banach $A$-bimodule is a special type of a generalized derivation.

\section{Preliminaries}

 Let $A$ be an algebra, $X$ a vector space over the same field and $V:A \times A \to X$ a bilinear map. We say that $V$ has \textit{product property} if there exists a linear map $\vf:A\to X$ such that for all $a,b\in A$
 \begin{align}\label{V(a_b)=vf(ab)}
     V(a,b)=\vf(ab).
 \end{align}
 It follows that
 \begin{align}\label{V(a_b)=V(1_ab)=V(ab_1)}
     V(a,b)=V(a',b'),\text{ if }ab=a'b'
 \end{align}
 for all $a,b,a',b'\in A$. If $A$ has $1$, then \cref{V(a_b)=V(1_ab)=V(ab_1)} implies \cref{V(a_b)=vf(ab)} with $\vf(a)=V(1,a)=V(a,1)$.

 If $A$ is a Banach algebra and $X$ is a Banach space, then a \textit{bounded} bilinear map $V:A\times A\to X$ is said to have \textit{product property} if there exists a \textit{bounded} linear map $\vf:A\to X$ such that \cref{V(a_b)=vf(ab)} holds.
 
 More generally, given $c\in A$, we say that $V$ has \textit{product property at $c$} if for all $a,b\in A$:
\begin{align*}%\label{ab=z=>V(ab)=V(1z)=V(z1)}
    ab=c \impl V(a,b)=\vf(ab),
\end{align*}
 where $\vf:A\to X$ is a linear map. This implies \cref{V(a_b)=V(1_ab)=V(ab_1)} subject to $ab=a'b'=c$. Algebras $A$ on which ``having product property'' is the same as ``having product property at zero'' are called \textit{zero product determined algebras} (zpd algebras)~\cite{Bresar-zpd}. If $A$ is a Banach algebra, then we say that $A$ is zero product determined (as a Banach algebra) if every bounded bilinear map that has products property at zero has product property.
 
 By \cite[Theorem 2.3]{Bresar-zpd} any algebra generated by idempotents, in particular $M_n(\CC)$, is zpd. On the other hand, every every C$^*$-algebra is a zpd Banach algebra (see~\cite[Section 5.3]{Bresar-zpd})
 
Observe that if $A$ is a Banach algebra with a bounded approximate identity and $c\in A$, then by \emph{Cohen factorization property} (see \cite[Theorem 1]{CFP}) there exist $a,b\in A$ such that $ab=c$. Thus, in this setting product property makes sense at any element $c\in A.$ Recall that every C$^*$-algebra has a bounded approximate identity (see \cite[Corollary I.7.5]{Takesaki}). We shall make use of these facts along the paper without an explicit mention.
 
 Let $A$ be an algebra and $c\in A.$ We say that $A$ is \textit{determined by products at $c$} if for every vector space $X$ and every bilinear map $V:A\times A\to X$ if $V$ has product property at $c$, then $V$ has product property. If $A$ is a Banach algebra then we say $A$ is determined by products at $c$ (as a Banach algebra) if every \textit{bounded} bilinear map that has products property at $c$ has product property. Clearly, if $A$ is a zpd algebra (resp. zpd Banach algebra) and every bilinear (resp. bilinear bounded) map that has product property at $c$ has product property at zero, then $A$ is determined by products at $c$. 
 %The analogous statement for Banach algebras is clear. 
 
 The next two examples show that  $M_n(\CC)$ is not determined by products at $c$ whenever $\rank(c)>n-2$.
 
  \begin{exm}\label{exm-rank(c)=n-1}
 Let $c=I_{n-1} \oplus 0_{1,1}$. Define 
 $V:M_n(\CC)\times M_n(\CC)\to \CC$ given by 
 $$
 V(a,b)=a_{nn}b_{nn},
 \text{ where }
 a=(a_{ij})
\text{ and }
b=(b_{ij}).
$$ A close look at \cite[Example 1]{Costara-Racsam-2022} allows one to realize that $V(a,b)=0$ whenever $ab=c$. Thus $V$ has the product property at $c.$ However, $V$ does not have product property at zero. Indeed, if $n=2$ take $a=\begin{pmatrix}
1 & 1\\
2&2
\end{pmatrix}$ and $b=\begin{pmatrix}
-1& -2\\
1&2
\end{pmatrix}$. Then $ab=0$, however $V(a,b)=4\neq 0. $ For $n\geq 3$ we let $a'=0_{n-2,n-2}\oplus a$ and $b'=0_{n-2,n-2}\oplus b$. Clearly, $a'b'=0$ but $V(a',b')=4\neq 0$. 
\end{exm}

\begin{exm}\label{exm-rank(c)=n}
Define $V:M_n(\CC)\times M_n(\CC)\to M_n(\CC)$ by $V(a,b)=a^tb^t=(ba)^t$. Then $V$ has product property at $I_n$, but does not have product property at zero. Indeed, it is enough to take $a,b\in M_n(\CC)$ such that $ab=0$ but $ba\neq 0$ (this is similar to \cite[Example 2.7]{Burgos-Cabello-Peralta19}). 
\end{exm}

Let $A$ be a C$^*$-algebra. An element $e\in A$ is called a \textit{projection} if $e^2=e$ and $e^*=e$. We denote by $P(A)$ the set of projections of $A.$ Furthermore, $u\in A$ is a \textit{partial isometry}~\cite{Polo-Peralta-Partial_Isom} if $uu^*$ is a projection (equivalently, $u^*u$ is a projection). A partial isometry $u$ is said to be \textit{minimal} if $uu^*Au^*u=\CC u.$ The set of non-zero minimal partial isometries of $A$ will be denoted by $MinPI(A)$.
An operator $u\in B(H)$ on a complex Hilbert space $H$ is a partial isometry if and only if it is an isometry between $\ker(u)^{\perp}$ (its \textit{initial space}) and $u(H)$ (its \textit{final space}), and $u$ is minimal if and only if $\dim(u(H))=1$. We say that $a,b\in A$ are \textit{orthogonal} (denoted $a\perp b)$ if $ab^*=b^*a=0$. Observe that a minimal partial isometry cannot be the sum of two non-zero mutually orthogonal partial isometries. For if $u=u_1+u_2$ with $u$ minimal and $u_1\perp u_2$ two partial isometries, then $uu^* u_1 u^*u= u_1\notin \CC u.$ We collect some well-known results on partial isometries for later use. 

Let $H$ be a complex Hilbert space. Given $e,f\in H$, define $e\otimes f \in B(H)$ by 
\begin{align*}
    (e\otimes f)(h):=\langle h,f\rangle e,
\end{align*}
where $\langle\cdot,\cdot\rangle$ denotes the scalar product in $H.$ Observe that in $\CC^n$ we have $e_i\otimes e_j=e_{ij}$, where $\{e_i\}_{i=1}^n$ is the natural basis of $\CC^n$ and $\{e_{ij}\}_{i,j=1}^n$ are the matrix units.

%\red{trocar referencias para \cite{Harris_An_Inv_1978}}
\begin{lem}\label{lem partial isometries}
Let $H$ be a complex Hilbert space. We have:
\begin{enumerate}
    \item\label{u=e-otimes-f} $u\in MinPI(B(H))$ iff $u=e\otimes f$ for some $e,f\in H$ with $\|e\|\cdot \|f\|=1$;
     \item\label{uv=0<=>f_1-perp-e_2} Let $u=e_1 \otimes f_1,v=e_2 \otimes f_2\in MinPI(B(H))$. Then $u v=0$ iff $f_1 \perp e_2$.
    \item\label{u-perp-v<=>e_1-perp-e_2-and-f_1-perp-f_2} Let $u=e_1 \otimes f_1,v=e_2 \otimes f_2 \in MinPI(B(H))$. Then $u \perp v$ iff $e_1 \perp e_2$ and $f_1 \perp f_2.$
    \item\label{a=sum_i=1^k-lb_iu_i} If $a\in B(H)$ and $\rank(a)=k<\infty$, then $a=\sum_{i=1}^k\lambda_i u_i$, where $\lambda_i>0$ and $u_i\in MinPI(B(H))$ for all $i$ with $u_i\perp u_j$ for $i\ne j$. Moreover, if $u_i=e_i\otimes f_i$, $1\le i\le k$, then $a(H)=\Span\{e_1,\dots,e_k\}$ and $\ker(a)^\perp=\Span\{f_1,\dots,f_k\}$. 
\end{enumerate}
\end{lem}

\begin{proof}
 \cref{u=e-otimes-f,u-perp-v<=>e_1-perp-e_2-and-f_1-perp-f_2} can be found in \cite[Section~2.1.1]{Tesis-Paco-Polo}. The decomposition $a=\sum_{i=1}^k\lambda_i u_i$ in \cref{a=sum_i=1^k-lb_iu_i} follows from \cite[Proposition~3.4]{Harris_An_Inv_1978}. The spaces $a(H)$ and $\ker(a)^\perp$ can be easily calculated by applying $a$ to an arbitrary $x\in H$ and using \cref{u=e-otimes-f,u-perp-v<=>e_1-perp-e_2-and-f_1-perp-f_2}.%Theorems 3.3 and 6.2 in \cite{Harris-JC*-triples-1981}. 
 
 Let us prove \cref{uv=0<=>f_1-perp-e_2}. Fix $h\in H.$ Then $uv(h)=(e_1 \otimes f_1)(e_2 \otimes f_2)h= \langle f_1,e_2 \rangle \langle h,f_2 \rangle e_1.$ Clearly, if $f_1 \perp e_2$ then $uv=0$. Conversely, if $f_1$ and $e_2$ are not orthogonal then $uv(f_2)=\|f_2\|^2 \langle f_1,e_2 \rangle e_1 \neq 0.$
\end{proof}

\begin{rem}\label{rem cstar is jbstar}
   Sometimes it will be convenient to realize C$^*$-algebras in a larger class of Banach spaces. 

   A \textit{Jordan-Banach} algebra is a Jordan algebra endowed with a complete norm $\|.\|$ satisfying $\|a \circ b\|\leq \|a\| \|b\|, \forall a,b\in A$. A Jordan-Banach algebra endowed with an involution $^*$ is a JB$^*$-algebra if it satisfies the additional axiom $\|U_a ({a^*}) \|= \|a\|^3, \ (a\in A)$ where for each $a\in A$ the mapping $U_a:A\to A$ is defined by $U_a (b) := 2(a\circ b)\circ a - a^2\circ b.$ Every C$^*$-algebra becomes a JB$^*$-algebra when endowed with the product 
   \begin{align*}
       a\circ b=\frac{1}{2}(ab+ba).
   \end{align*}

   A \emph{JB$^*$-triple} is a complex Banach space $E$ equipped with a triple product  $\{\cdot,\cdot,\cdot\}:E\times E\times E\rightarrow E$ which is linear and symmetric in the outer variables, and conjugate linear in the middle one, and satisfies the axioms:
\begin{itemize}
	\item (Jordan identity) for $a,b,x,y,z$ in $E$,
	$$\{a,b,\{x,y,z\}\}=\{\{a,b,x\},y,z\}
	-\{x,\{b,a,y\},z\}+\{x,y,\{a,b,z\}\};$$
	\item $L(a,a):E\rightarrow E$ hermitian with non negative spectrum, where $L(a,b)(x)=\{a,b,x\}$ with $a,b,x\in
	E$;
	\item $\|\{x,x,x\}\|=\|x\|^3$ for all $x\in E$,
\end{itemize}
   
  Every C$^*$-algebra is a JB$^*$-triple when endowed with the triple product \begin{align*}
        \{a,b,c\}=\frac{1}{2}(ab^*c+cb^*a)
    \end{align*}
    while a JB$^*$-algebra also becomes a JB$^*$-triple with 
    \begin{align*}
        \{a,b,c\}=(a\circ b^*)\circ c+(c\circ b^*)\circ a-(a\circ c)\circ b^* .
    \end{align*}
    See \cite{Cabrera_Palacios_Book} for more details.
    \end{rem}

\section{Product property in compact C$^*$-algebras}\label{sec-compact}

\begin{defn}\label{defn-zp-dense}
Let $A$ be a normed algebra and $S\subseteq A$. We say that $S$ is \textit{zp-dense} in $A$ if
\begin{enumerate}
 \item\label{span(S)=A}  $\overline{\Span(S)}^{\|.\|}=A ;$
 \item\label{exists-v-with-uv=0} for each $u\in S$ there exists  $v\in S$ such that $uv=0$;
 \item\label{xy=0<=>u_iv_j=0} any $x,y\in A$ with $xy=0$ can be expressed as $x=\sum_{n=1}^{\infty} \alpha_n u_n$ and $y=\sum_{n=1}^{\infty} \beta_n v_n$ with $u_i,v_j\in S$ and $u_i v_j=0$ for all $i,j.$
 \end{enumerate}
 %Let us fix $c\in A.$ We say that $c$ \emph{factorizes through zero products} in $S$ if for every $u,v \in S$ with $uv=0$ there exist $a,b\in A$ such that $av=0, ub=0$ and $ab=c.$
\end{defn}

\begin{defn}\label{defn-fact-through-c}
Let $A$ be an algebra, $c\in A$ and $x,y\in A$ with $xy=0$. We say that $c$ \emph{factorizes through $(x,y)$} if there exist $a,b\in A$ such that $ay=0, xb=0$ and $ab=c.$ 

 More generally, let $S\subseteq A$.  We say that $c$ \emph{factorizes through zero products in $S$}  if, for every $x,y \in S$ with $xy=0$, $c$ factorizes through $(x,y)$. 
\end{defn}

\begin{rem}%\red{is this remark needed in section 3??}
    If $c$ factorizes through zero products in $S$ and the set $S$ satisfies \cref{defn-zp-dense}\cref{exists-v-with-uv=0}, then for every $x\in S$ there exist $a,b\in A$ such that $xb=0$ and $ab=c$. 
\end{rem}

%{prop Jilius dec implies zpd} {prod-prop-c=>prod-prop-0-for-A+B}

\begin{lem}\label{lem fact thru zp xy=0 implies V(x_y)=0}
     Let $A$ be an algebra, $X$ a vector space and $V:A\times A\to X$ a bilinear map. Suppose that $V$ has product property at $c\in A.$ For all $x,y\in A$ such that $c$ factorizes through $(x,y)$ we have $V(x,y)=0.$
    % \begin{align*}
     %    xy=0\impl V(x,y)=0.
    % \end{align*}
\end{lem}
\begin{proof}
There exist $a,b\in A$ such that $xb=0,ay=0$ and $ab=c.$ Thus, we have $ab=(x+a)b=a(b+y)=(a+x)(b+y)=c.$ Successive applications of \cref{V(a_b)=V(1_ab)=V(ab_1)} to the latter equalities show that $V(x,y)=0.$
\end{proof}
\begin{prop} \label{prop Jilius dec implies zpd}
 Let $A$ be a normed algebra and $S\sst A$ zp-dense. Assume that $c\in A$ factorizes through zero products in $S$.
 % \begin{enumerate}
 %     \item\label{prod-prop-c=>prod-prop-0-for-A} 
     If a bounded bilinear map $V:A\times A \to X$, where $X$ is a normed space, has product property at $c$, then $V$ has product property at zero.
 % \item\label{prod-prop-c=>prod-prop-0-for-A+B} If a bounded bilinear map  $V:(A\oplus B) \times (A\oplus B)\to X$, where $B$ is another normed algebra and $X$ is a normed space, has product property at $c,$ then $V$ has product property at zero.
 % \end{enumerate}
\end{prop}
\begin{proof}
 % \textit{\cref{prod-prop-c=>prod-prop-0-for-A}.} 
 Given $u,v\in S$ with $uv=0$ we have $V(u,v)=0$ by \cref{lem fact thru zp xy=0 implies V(x_y)=0}.  Now take $x,y\in A$ such that $xy=0.$ Since $S$ is zp-dense in $A$, then $x=\sum_{n=1}^{\infty} \alpha_n u_n,\; y=\sum_{n=1}^{\infty} \beta_n v_n $ for $\alpha_i,\beta_j\in \CC$ and $u_i,v_j\in S$ with $u_iv_j=0$ for all $i,j$. By the first part of the proof, we have $V(u_i,u_j)=0.$ Finally, by continuity and bilinearity of $V$ we have 
 $$V(x,y)=V\left(\sum_{n=1}^{\infty} \alpha_n u_n,\sum_{k=1}^{\infty} \beta_k v_k \right)=\sum_{n} \sum_k \alpha_n \beta_k V(u_n,v_k)=0 .$$ Thus $V$ has product property at zero. 
% 
 % \textit{\cref{prod-prop-c=>prod-prop-0-for-A+B}.}
 %   Let $B$ be a  Banach algebra and $V:(A\oplus B)\times (A\oplus B)\to X $ be a bounded bilinear map that has product property at $c\in A.$ Let $x,y\in A\oplus B$ such that $xy=0.$ We distinguish five cases:
 %  
 %   \textit{Case 1.} If $x,y\in A$, then $V(x,y)=0$ by \cref{prod-prop-c=>prod-prop-0-for-A}, since $V$ has product property at zero when restricted to $A\times A$.
 %  
 %   \textit{Case 2.} If $x,y \in B$, then take any $a,b\in A$ with $ab=c$. Since $xb=0$ and $ay=0$ it follows that $(a+x)b=c, a(b+y)=c$ and $(a+x)(b+y)=c$. Again, several applications of \cref{V(a_b)=V(1_ab)=V(ab_1)} allow to show that $V(x,y)=0.$
 %  
 %   \textit{Case 3.} $x\in A$ and $y\in B.$ Let us first suppose that $x\in S$. By the observation after \cref{defn-fact-through-c} we can find $a,b\in A$ such that $xb=0$ and $ab=c.$ Since $y A=Ay=0$ we also have $ya=ay=yb=by=0.$ Thus we have  $(a+x)b=c, a(b+y)=c$ and $(a+x)(b+y)=c$ which allows to show that $V(x,y)=0.$ Now, in the general case write  $x=\sum_{n=1}^{\infty} \alpha_n u_n$ where $u_n\in S\subset A$ and hence $u_n y=y u_n=0$ for all $n.$ We have $V(x,y)=\sum_{n=1}^{\infty} \alpha_n V(u_n,y)=0$ by continuity of $V$.
  % 
 %   \textit{Case 4.} $x\in B$ and $y\in A.$ This is similar to Case 3.
 %  
 %   \textit{Case 5.} $x=x_1+x_2$ and $y=y_1+y_2$ with $x_1,y_1\in A$, $x_2,y_2\in B$ and $xy=0.$ Then $x_1 y_1=x_2 y_2=x_1y_2=x_2y_1=0$. If follows from the previous cases that $V(x,y)=V(x_1,y_1)+V(x_1,y_2)+V(x_2,y_1)+V(x_2,y_2)=0.$
\end{proof}

  A basic example of the set $S$ from \cref{prop Jilius dec implies zpd} is $S=MinPI(A)$, where $A$ is a C$^*$-algebra having ``enough'' minimal partial isometries.

Given a family $(X_{\lambda})$ of Banach spaces, let $\bigoplus_{\lb}X_{\lb}$ be the set of all $(x_{\lb})\in \prod_{\lb} X_{\lb}$ such that $\|(x_\lb)\|:=\sup \{ \|x_\lb\|\}<\infty.$ Then $\bigoplus_{\lb}X_{\lb}$ is a Banach space. Moreover, if each $X_{\lambda}$ is a C$^*$-algebra, then $\bigoplus_{\lb}X_{\lb}$ is also a C$^*$-algebra. The algebra $\bigoplus_{\lb}X_{\lb}$ contains the subalgebra $\bigoplus_{\lb}^{c_0}X_{\lb}$ formed by $(x_{\lb})\in \prod_{\lb} X_{\lb}$ with $\lim \|x_\lb\|=0$.

Let $A$ be a C$^*$-algebra. The \textit{socle} of $A$, denoted $soc(A)$, is defined as the linear span of all minimal projections in $A$. The ideal of compact elements in $A$, denoted $K (A)$, is defined as the norm closure of $soc(A)$. A C$^*$-algebra is said to be \textit{compact} or dual if $A = K (A)$. In this case, it is known by \cite[Theorem 8.2]{Alexander-Compact-banah} that $A\cong \bigoplus_{\lb}^{c_0} K(H_{\lb})$ where each $H_{\lb}$ is a complex Hilbert space and $K(H_\lb)$ is the algebra of compact operators on $H_\lb$. It is worth noticing that if $\dim(H)=\infty$ then $K(H)$ is not unital. If $A=K(H)$ then $soc(A)$ coincides with the set of finite-rank operators on $H$. We denote by $\pi_{\lb}:A \to K(H_{\lb})$ the projection onto $K(H_{\lb}).$ For the basic references on compact C$^*$-algebras see \cite{Alexander-Compact-banah,Kaps-Dual-C*,Ylinen-Compact-C*}.  

Every compact C$^*$-algebra is a compact JB$^*$-triple in the sense of \cite{BuChu}. A \textit{tripotent} in a JB$^*$-triple $E$ is an element $e\in E$ such that $\{e,e,e\}=e$. When $E$ is a C$^*$-algebra regarded as a JB$^*$-triple then $e$ is a tripotent whenever $ee^*e=e$ which is equivalent to being a partial isometry.

\begin{lem}\label{MinPi(A)=union-MinPI(A_lb)}
Let $A=\bigoplus_{\lb}^{c_0} A_{\lb}$, where each $A_{\lb}$ is a C$^*$-algebra. Then $MinPI(A)=\bigcup_{\lb} MinPI(A_{\lb}).$ 
\end{lem}
 \begin{proof}
     Let $u\in MinPI(A)$. We first prove that $u$ lies in a unique $A_\lb.$ For, take $\lb$ such that $\pi_{\lb}(u)\neq 0.$ Then 
     \begin{align*}
     \pi_{\lb}(u)^*(u-\pi_\lb(u))=\pi_\lb(u^*)u-\pi_\lb(u^*)\pi_\lb(u)=\pi_\lb(u^*)\pi_\lb(u)-\pi_\lb(u^*)\pi_\lb(u)=0,
     \end{align*}
 and similarly $(u-\pi_\lb(u))\pi_\lb(u)^*=0$, so $\pi_\lb(u)\perp (u-\pi_\lb(u)).$  Since $\pi_\lb(u)$ is a partial isometry and $u=(u-\pi_\lb(u))+\pi_\lb(u)$ then $u-\pi_\lb(u)$ is a partial isometry. Hence, if $u\not\in A_\lb$, then $u$ is the sum of the two non-zero mutually orthogonal partial isometries $u-\pi_\lb(u)$ and $\pi_\lb(u)$ which contradicts the minimality of $u$. Thus, $MinPI(A)\sst \bigcup_\lb MinPI(A_\lb).$ For the converse inclusion, observe that if $u\in MinPI(A_\lb)$ for some $\lb$ and $x\in A$ then $uu^*xu^*u=uu^*\pi_\lb(x)u^*u=\alpha u$ for some $\alpha \in \CC,$ thus $uu^*Au^*u\subseteq \CC u.$ Taking $x=\alpha u, \alpha \in \CC$, we have  $uu^*Au^*u= \CC u.$
 \end{proof}
 
 \begin{lem}\label{lem S zp-dense in compact C*} Let $A=\bigoplus_{\lb}^{c_0} K(H_{\lb})$ be a compact C$^*$-algebra of dimension $n>1$. Then $MinPI(A)$ is zp-dense in $A.$
 \end{lem}
 \begin{proof}
  Set $S=MinPI (A).$ By \cite[Remark 4.6]{BuChu},  $\Span(S)$ is norm dense in $A$, so \cref{defn-zp-dense}\cref{span(S)=A} holds. As to, \cref{defn-zp-dense}\cref{exists-v-with-uv=0}, by \cref{MinPi(A)=union-MinPI(A_lb)} every $u\in S$ belongs to $K(H_\lb)$ for some $\lb$. If there exists another direct summand $K(H_\mu)$ of $A$, then one can choose $v\in K(H_\mu)\cap S$, so that $uv=0$. Otherwise, $A=K(H)$ for some Hilbert space $H$ of dimension $>1$ and $u=e\otimes f$ for some $e,f\in H$ with $\|e\|=\|f\|=1$, so choosing $g\in H$ with $g\perp f$ (which is possible, as $\dim(H)>1$) and defining $v=g\otimes g\in S$, we have $uv=0$ by \cref{lem partial isometries}\cref{uv=0<=>f_1-perp-e_2}.
  
  The proof of \cref{defn-zp-dense}\cref{xy=0<=>u_iv_j=0} will be split into two cases.
  
 \textit{Case 1.} $A=K(H)$. Take $x,y\in K(H)$. By \cite[Remark 4.5]{BuChu} and \cref{lem partial isometries}\cref{u=e-otimes-f}\cref{u-perp-v<=>e_1-perp-e_2-and-f_1-perp-f_2} we have $x=\sum_n \alpha_n f_n\otimes g_n$ and $y=\sum_n \beta_n h_n\otimes k_n$ with $\{f_n\}, \{g_n\},\{h_n\} ,\{k_n\}\sst H$ orthonormal. Assume that $xy=0.$ Then $y(H)\subseteq \ker(x),$ so $\Span\{h_n\}\subseteq y(H)\subseteq \overline{\Span\{g_n\}}^{\perp}$ and hence $h_i \perp g_j$ for all $i,j$ which implies $(f_i\otimes g_i) (h_j\otimes k_j)=0 $ for all $i,j.$ 
 
 \textit{Case 2.} $A=\bigoplus_{\lambda}^{c_0} K(H_{\lambda})$. Take $x,y\in A$. By \cite[Remark 4.5]{BuChu} $x=\sum_{n=1}^{\infty} \alpha_n u_n,\; y=\sum_{n=1}^{\infty} \beta_n v_n $ with $u_n,v_n\in S,(\alpha_n),(\beta_n)\subseteq \mathbb{R}^{+}$,  $\{u_n\}$ and $\{v_n\}$ mutually orthogonal.
 
 In this case $xy=0$ if and only if $\pi_{\lambda} (x)\pi_{\lambda}(y)=0$ for all ${\lambda}.$ Let us fix ${\lambda}.$ Then $\pi_{\lambda}(x)=\sum_{n=1}^{\infty} \alpha_n \pi_{\lambda}(u_n),\; \pi_{\lambda}(y)=\sum_{n=1}^{\infty} \beta_n \pi_{\lambda}(v_n) $ where $\pi_{\lambda}(u_n)$ is either $u_n$ (if $u_n\in A_{\lambda}$) or $0$, and the same holds for $\pi_{\lambda}(v_n)$. Applying Case 1 to $\pi_{\lambda}(x),\pi_{\lambda}(y)\in K(H_\lb)$, we conclude that $\pi_\lb(u_nv_m)=\pi_{\lambda}(u_n)\pi_{\lambda}(v_m)=0$ for all $n,m.$ Since $\lb$ was arbitrary, $u_nv_m=0$.
 % Now we take take $n,m\in \mathbb{N}.$ If $u_n$ and $v_m$ lie in different summands of $A$ then $u_nv_m=0.$ Otherwise, there exists a unique ${\lambda}$ such that $\pi_{\lambda}(u_n)=u_n$ and $\pi_{\lambda}(v_m)=v_m.$ As we have proved in the above paragraph, in this case we have $u_n v_m=0.$ 
 % We have shown that $uv=0$ implies $u_n v_m=0$ for all $n,m$ which gives  \cref{defn-zp-dense}\cref{xy=0<=>u_iv_j=0}.
 \end{proof}
 
 %Theorems 3.3 and 6.2 in \cite{Harris-JC*-triples-1981}
 \begin{defn}
     Let $A=\bigoplus_\lb^{c_0} K(H_\lb)$ be a compact C$^*$-algebra and $c\in A$. We say that $c$ satisfies the \textit{rank hypothesis}  if 
     \begin{enumerate}[label=(RH)]
         \item\label{RH} $\rank(\pi_\lb(c) )\le\dim(H_\lb)-2$ for all $\lb$.
     \end{enumerate}
 \end{defn}

\begin{lem}\label{lem compact star factor thu zp}
Let $A=\bigoplus_\lb^{c_0} K(H_\lb)$ be a compact C$^*$-algebra and $c\in A$ satisfying \cref{RH}. Then $c$ factorizes through zero products in $MinPI(A).$
\end{lem}
 \begin{proof}
    Let us fix two minimal partial isometries $u,v\in A$ such that $uv=0.$ We are going to find $a,b\in A$ such that $av=ub=0$ and $ab=c.$ 
    
    \textit{Case 1}. $A=K(H).$  By \cref{u=e-otimes-f,uv=0<=>f_1-perp-e_2} of \cref{lem partial isometries} there exist norm-one vectors $x_1,x_2,y_1,y_2\in H$  such that $u=x_1\otimes x_2$, $v=y_1 \otimes y_2$ and $x_2\perp y_1$.
  
   By \cite[Remark 4.5]{BuChu} and \cref{lem partial isometries}\cref{u=e-otimes-f}  there exist $\{\alpha_n\}\in c_0^+$  and 
   %a family of mutually orthogonal minimal partial isometries $(u_n) \subset A$ such that $c=\sum_k \lambda_k u_k.$ By 
    orthonormal $\{e_n\},\{f_n\}\sst H$ such that 
    $$
    c=\sum_{n=1}^{\infty} \alpha_n e_n\otimes f_n.
    $$ 
    If $r=\dim(H)<\infty$, set $k=\rank(c)$ and $I=\{1,\dots,k\}$. It follows that $\dim( \{x_2,y_1\}^{\perp})=r-2.$
  Recall that by hypothesis $k\leq r-2$, thus we can take an orthonormal  set $\{h_n\}_{n\in I} \sst \{x_2,y_1\}^{\perp}.$ If $\dim(H)=\infty$, we let $I=\mathbb{N}$ and take $\{h_n\}_{n\in I} \subseteq \{x_2,y_1\}^{\perp}$ an orthonormal system. 
 Let 
   $$
   a=\sum_{n\in I} \alpha_n^{\frac{1}{2}} (e_n\otimes h_n).
   $$
    Observe that the series above converges by \cite[Remark 7]{BurFerGarPer_Aut_Cont_2011}. Since $v(H)=\CC y_1$ and $y_1\perp h_i$, $i\in I$, we have $av=0.$ Now define 
   $$
   b=\sum_{n\in I} \alpha_n^{\frac{1}{2}}( h_n\otimes f_n).
   $$
   Since $\ker(u)= \{x_2\}^{\perp}\supseteq\{h_n\}_{n\in I}$, it is clear that $ub=0$ holds. Moreover, for any $x\in H$ we have
   $$
   (e_i\otimes h_i)(h_j\otimes f_j)x=\langle x,f_j\rangle(e_i\otimes h_i)h_j=\langle x,f_j\rangle\langle h_j,h_i\rangle e_i=\langle x,f_j\rangle\delta_{ij} e_i=\delta_{ij}(e_i\otimes f_j)x,
   $$
   whence
   $$
   ab=\sum_{n\in I}^{\infty} \alpha_n e_n\otimes f_n=c.
   $$
   Thus, $c$ factorizes through zero products in $S.$

 \textit{Case 2.} $A=\bigoplus_\lb^{c_0} K(H_\lb)$, where each $H_\lb$ is a complex Hilbert space. The proof is divided into several cases.
   
   \textit{Case 2.1.} $u,v\in K(H_\mu).$  If $\pi_\mu(c)=0$ then $c\in \bigoplus_{\lb\neq \mu}^{c_0} K(H_\lb)$ and for any $a,b \in K(H_\mu)^{\perp}$ such that $ab=c$ we have $av=ub=0$. Let us suppose that $ \pi_\mu(c)\neq 0.$ By the assumptions on $c$ and Case 1, there exist $a_\mu,b_\mu\in K(H_\mu)$ such that $a_\mu v=u b_\mu=0$ and $a_\mu b_\mu=\pi_\mu(c).$ We can also find $a,b\in K(H_\mu)^{\perp}$ such that $ab=c-\pi_\mu(c).$  We have $(a+a_\mu)v=u(b+b_\mu)=0 $ and $(a+a_\mu)(b+b_\mu)=c.$
   
  \textit{Case 2.2.} $u\in K(H_\mu),v \in K(H_\nu)$ with $\mu\neq \nu$ such that $\pi_\mu(c)=0$ and $\pi_\nu(c)=0$. We can find $a,b\in \bigoplus_{\lb \neq \mu,\nu}^{c_0} K(H_\lb)$ such that $ab=c$. Clearly, $av=0$ and $ub=0$ hold.

 \textit{Case 2.3.} $u\in K(H_\mu),v \in K(H_\nu)$ such that  $\pi_\mu(c) \neq  0$ and $\pi_\nu(c)=  0$. We can find $v'\in MinPI(K(H_\mu))$ such that $uv'=0.$ There exist $a_\mu,b_\mu\in K(H_\mu)$ such that $ub_\mu=0,a_\mu v'=0$ and $a_\mu b_\mu =\pi_\mu(c)$. There also exist $a',b'\in \bigoplus_{\lb \neq \mu,\nu}^{c_0} K(H_\lb)$ such that $a'b'=c-\pi_\mu(c)$. We have $u(b_\mu+b')=0,(a_\mu+a')v=0$ and $(a_\mu+a')(b_\mu+b')=c.$

 \textit{Case 2.4.}  $u\in K(H_\mu),v \in K(H_\nu)$ such that  $\pi_\mu(c)\neq 0$, $\pi_\nu(c)\neq  0$ and $\mu\ne \nu$. There exist $a,b\in \big(K(H_\mu)\oplus K(H_\nu)\big)^{\perp}$ such that $ab=c-\pi_\mu(c)-\pi_\nu(c).$  We can also find $a_\mu,b_\mu\in K(H_\mu), a_\nu,b_\nu\in K(H_\nu)$ such that $u b_\mu=0,$ $a_\mu b_\mu=\pi_\mu(c),$ $b_\nu v=0$ and $a_\nu b_\nu =\pi_\nu (c).$ We have $(a+a_\mu+a_\nu )v=u(b+b_\mu+b_\nu)=0$ and  $(a+a_\mu+a_\nu)(b+b_\mu+b_\nu)=c.$
 %We have shown that given two minimal partial isometries $u,v\in A$ such that $uv=0$ there exist $a,b\in A$ such that $av=ub=0$ and $ab=c.$
 \end{proof}

% Now we are in the position to obtain a non-unital version of  \cite{Julius-JMAA-22}[proposition 2.1]
 
\begin{thrm}\label{thm main compact c*}
 A compact C$^*$-algebra $A=\bigoplus_\lb^{c_0} K(H_\lb)$ is determined (as a Banach algebra) by products at any element $c\in A$ satisfying \cref{RH}. In particular, if $A$ has no finite-dimensional summands, then $A$ is determined by products at any element (as a Banach algebra).
\end{thrm}
\begin{proof}
By \cref{lem S zp-dense in compact C*} the set $S=MinPI(A)$ is zp-dense in $A$, and by \cref{lem compact star factor thu zp} the element $c$ factorizes through zero products in $S$. Let $X$ be a Banach space and $V:A\times A\to X$ be a bounded bilinear map that has product property at $c$. By \cref{prop Jilius dec implies zpd} the map $V$ has product property at zero. Now by Theorem 5.19 together with Proposition 4.9 from  \cite{Bresar-zpd} there exists a bounded linear map $\vf:A\to X$ such that $V(ab)=\vf(ab)$ for all $a,b\in A.$
\end{proof}

\section{Product property at compact elements in von Neumann algebras}\label{sec-von-Neumann}

 Now we turn our attention to von Neumann algebras. We start with the case of bilinear maps that have product property at a finite-rank element. In this case, continuity can be relaxed.

 \begin{defn}\label{defn-generalized-fact-through-c}
  Let $A$ be an algebra, $c\in A$ and $x,y\in A$ with $xy=0$. We say that $c$ \emph{has a generalized factorization through $(x,y)$} if one of the following holds. 
  \begin{enumerate}
      \item\label{def gen fact 2} There exist $x',x''\in A$ such that $x=x'+x''$ and $c$ factorizes through $(x',y)$ and $(x'',y)$.
      \item\label{def gen fact 1} There exist $y',y''\in A$ such that $y=y'+y''$ and $c$ factorizes through  $(x,y')$ and $(x,y'')$.
  \end{enumerate}
 Let $S\subseteq A$ and $c\in A.$ We say that $c$ \emph{has a generalized factorization through zero products in $S$} if $c$ has a generalized factorization through all $(x,y)$ with $x,y\in S$ and $xy=0$. 
\end{defn}
 
 The proof of the following lemma is a part of the proof of \cite[Proposition 2.1]{Julius-JMAA-22}.
 
 \begin{lem}\label{lem technical Julius}
 Let $X$ be an infinite-dimensional Banach space and $c\in B(X)$ a finite-rank operator. Then $c$ has generalized factorization through zero products in $B(X)$.
 \end{lem} 

\begin{lem}\label{lem gen zp factoriz implies pp at zero}
Let $A$ be an algebra,  $X$ a vector space and $V:A\times A\to X$ a bilinear map. Let $c\in A$ and $x,y\in A$ such that $xy=0$. Suppose that  $c$ has generalized factorization through $(x,y)$. Then $V(x,y)=0$.
\end{lem}

 \begin{proof}
 \textit{Case 1.} \cref{defn-generalized-fact-through-c}\cref{def gen fact 2} holds.
    There exist $x',x''\in A$ such that $x=x'+x''$  and $c$ factorizes through  $(x',y)$ and $(x'',y)$. It follows from \cref{lem fact thru zp xy=0 implies V(x_y)=0} that $V(x',y)=V(x'',y)=0$ and hence $V(x,y)=V(x',y)+V(x'',y)=0$.
    
\textit{Case 2.} \cref{defn-generalized-fact-through-c}\cref{def gen fact 1} holds. An analogous argument yields $V(x,y')=V(x,y'')=0$ and $V(x,y)=V(x,y')+V(x,y'')=0.$
 \end{proof}

%  \begin{lem}\label{lem factoriz zp isomoprhism}
%     Let $\phi:A\to B$ be an isomorphism between two algebras and $c\in A$. Then 
%     \begin{enumerate}
%         \item $c$ factorizes trough $xy=0$ iff $\phi(c)$ factorizes trough $\phi(x)\phi(y)=0$,
%          \item $c$ has generalized factorizations trough $xy=0$ iff $\phi(c)$ has generalized factorizations  trough $\phi(x)\phi(y)=0$,
%          \item $c$ factorizes (respectively, has generalized factorizations) trough zero products in $S\subseteq A$ iff $\phi(c)$ factorizes (respectively, has generalized factorizations) trough zero products in $\phi(S)$.
%     \end{enumerate}
% \end{lem}

\begin{lem}\label{lem gener fact direct sums}
Let $A_1,\dots, A_n$ be unital algebras, $0\ne c=(c_1,\dots, c_n)\in A_1 \oplus \dots \oplus A_n $ and for every $i$ let $S_i\subseteq A_i$ such that if $c_i\neq 0$ then $c_i$ has generalized factorization through zero products in $S_i$. We have
\begin{enumerate}
   \item\label{lem gener fact direct sums 1}  $c$ has generalized factorization through zero products in every $S_i$;
    \item\label{lem gener fact direct sums 2} if $n\ge 2$, $i\neq j$ and $x\in  S_i,y\in S_j$ such that $Ann_r(x)\cap S_i\neq \{0\}, Ann_l(y)\cap S_j\neq \{0\}$, then one of the following holds:
    \begin{enumerate}
        \item\label{c-gen-factor-through-(x_y)}  $c$ has generalized factorization through $(x,y)$;
        \item\label{c-factor-through-(x'_y')-etc} there exist  $x',x''\in A_i, y',y''\in A_j$ such that $x=x'+x'',y=y'+y''$ and $c$ factorizes through $(x',y'), (x',y''), (x'',y')$ and $(x'',y'')$.
    \end{enumerate}
   \end{enumerate}
   \end{lem}
\begin{proof}
Throughout the proof we will use the following short notations: given $z\in A$, for any $1\le i\le n$ write $z_{\neq i}:=z-z_i$, and for any pair $1\le i\ne j\le n$ write $z_{\neq i,j}:=z-z_i-z_j$. For short, write also $1_{\neq i}:=1_A-1_{A_i}$ and $1_{\neq i,j}:=1_A-1_{A_i}-1_{A_j}$.

\cref{lem gener fact direct sums 1}. Without loss of generality we can assume that $n>1.$ Fix $i$ and take $x,y\in S_i$ such that $xy=0$. 
%For each $k\neq i$ with $c_k\neq 0$ take $a_k,b_k\in A_k$ such that $a_k b_k=c_k.$ For $k\neq i$ such that $c_k=0$ we set $a_k=b_k=0.$

If $c_i=0$ then we have 
 $$
1_{\neq i}y=xc=0 \mbox{ and } 1_{\neq i}c=c_{\neq i}=c.
 %\left(\sum_{k\neq i} a_k\right) y=x \left(\sum_{k\neq i}  b_k\right)=0 \mbox{ and } \left(\sum_{k\neq i}  a_k\right)\left(\sum_{k\neq i}  b_k\right)=\sum_{k\neq i}  c_k=c.
 $$
 Thus in this case $c$ factorizes through $(x,y)$.
 
Now suppose that $c_i\neq 0.$ In this case we use generalized factorization of $c_i$ through $(x,y)$ inside $A_i.$ 
  
%\textit{Case 1}  There exist $a_{i},b_{i}\in B(H_i)$ such that $a_{i} y=0,x b_{i}=0$ and $a_{i} b_{i} =c_i.$ We have $(\sum_k a_k ) y=0,x(\sum_k b_k)=0$ and $(\sum_k a_k ) (\sum_k b_k ) =c$.

\textit{Case 1.}  There exist $a_{1},a_{2},b_{1},b_{2},x',x''\in A_i$  such that $x=x'+x''$, $a_{1} y=a_{2} y=x'b_{1}=x''b_{2}=0$ and $a_{1} b_{1} =a_{2} b_{2}=c_i.$ We have 
\begin{align*}
(a_{1}+1_{\neq i} ) y=(a_{2}+1_{\neq i} ) y=x'(b_{1}+c_{\neq i})=x''(b_{2}+c_{\neq i})=0    
\end{align*}
and 
\begin{align}\label{(a_1+1_neq i)(b_1+c_neq i)=c}
    (a_{1}+1_{\neq i} ) (b_{1}+c_{\neq i}) =(a_{2}+1_{\neq i} ) (b_{2}+c_{\neq i})=c.    
\end{align}

\textit{Case 2}. There exist  $a_{1},a_{2},b_{1},b_{2},y',y''\in A_i$  such that $y=y'+y''$, $a_{1} y'=a_{2} y''=xb_{1}=xb_{2}=0$ and $a_{1} b_{1} =a_{2} b_{2}=c_i.$ Then
\begin{align*}
    (a_{1}+1_{\neq i} ) y'=(a_{2}+1_{\neq i}) y''=x(b_{1}+c_{\neq i})=x(b_{2}+c_{\neq i})=0,
\end{align*}
and the same equality \cref{(a_1+1_neq i)(b_1+c_neq i)=c} holds.
% \begin{align*}
%     (a_{1}+1_{\neq i} ) (b_{1}+c_{\neq i}) =(a_{2}+1_{\neq i} ) (b_{2}+c_{\neq i})=c.
% \end{align*}

In any case, $c$ has  generalized factorization through $(x,y).$

\cref{lem gener fact direct sums 2}. Fix $i\ne j$ and $x\in  S_i,y\in S_j$ with $Ann_r(x)\cap S_i\neq \{0\}, Ann_l(y)\cap S_j\neq \{0\}$.

%Take $a_k,b_k\in A_k$ such that $a_kb_k=c_k$ for every $k\notin \{i,j\}.$ 

If $c_i=c_j=0$, then $c$ factorizes through $(x,y)$, because 
$$
1_{\neq i,j} y=x c_{\neq i,j}=0 \mbox{ and } 1_{\neq i,j}c_{\neq i,j}=c.
$$

If $c_i\neq 0$ and $c_j=0$, then take $z\in Ann_r(x)\cap S_i, z\neq 0$. We use generalized factorization of $c_i$ through $(x,z)$ inside $A_i$. 

%\textit{Case 1}  There exist $a_{i},b_{i}\in A_i$ such that $a_{i} z=0,x b_{i}=0$ and $a_{i} b_{i} =c_i.$ We have $(\sum_{k\neq j} a_k ) y=0,x(\sum_{k\neq j} b_k)=0$ and $(\sum_{k\neq j} a_k ) (\sum_{k\neq j} b_k ) =c$.

\textit{Case 1}.  There exist $a_{1},a_{2},b_{1},b_{2},x',x''\in A_i$  such that $x=x'+x''$, $a_{1} z=a_{2} z=x' b_{1}=x'' b_{2}=0$ and $a_{1} b_{1} =a_{2} b_{2}=c_i.$ We have
\begin{align*}
  (a_{1}+1_{\neq i,j}) y=(a_{2}+1_{\neq i,j} ) y=x'(b_{1}+c_{\neq i,j})=x''(b_{2}+c_{\neq i,j})=0  
\end{align*}
 and 
 \begin{align}\label{(a_1+1_neq i_j)(b_1+c_neq i_j}
  (a_{1}+1_{\neq i,j} ) (b_{1}+c_{\neq i,j}) =(a_{2}+1_{\neq i}  ) (b_{2}+c_{\neq i,j} )=c.   
 \end{align}

\textit{Case 2}. There exist  $a_{1},a_{2},b_{1},b_{2},z',z''\in A_i$ such that $z=z'+z''$ and  $a_{1}z'=a_{2}z''=x b_{1}=x b_{2}=0.$ 
We have 
\begin{align*}
  (a_{1}+1_{\neq i,j}  ) y=(a_{2}+1_{\neq i,j} ) y=x(b_{1}+c_{\neq i,j} )=x(b_{2}+c_{\neq i,j})=0,  
\end{align*}
 and the same equality \cref{(a_1+1_neq i_j)(b_1+c_neq i_j} holds.
 % \begin{align*}
 %  (a_{1}+1_{\neq i,j}  ) (b_{1}+c_{\neq i,j} )=   (a_{2}+1_{\neq i,j}  ) (b_{2}+c_{\neq i,j} )=c
 % \end{align*}

Thus, whenever $c_i\neq 0$ and $c_j=0$, we fall into \cref{c-gen-factor-through-(x_y)}.

For the rest of the proof we suppose that $c_i\neq 0,c_j\neq 0$.

 We take $z\in Ann_r(x)\cap S_i, z\neq 0 $  and $w\in Ann_l(y)\cap S_j, w\neq 0 .$ We use  generalized factorization (inside $A_i$) of $c_i$ through $(x,z)$ to split the proof in cases, and generalized factorization (inside $A_j$) of $c_j$ through $(w,y)$ to split each case in sub-cases.

 \textit{Case 3}. There exist $a_{i,1},a_{i,2},b_{i,1},b_{i,2},x',x''\in A_i$ with $x=x'+x''$, such that $a_{i,1}z=a_{i,2}z=x'b_{i,1}=x''b_{i,2}=0$ and  $a_{i,1} b_{i,1}=a_{i,2}b_{i,2}=c_i$.

  \textit{Case 3.1}.   There exist $a_{j,1},a_{j,2},b_{j,1},b_{j,2}, y',y''\in A_j$  such that $y=y'+y''$, $a_{j,1}y'=a_{j,2}y''=w b_{j,1}=wb_{j,2}=0$ and $a_{j,1} b_{j,1}=a_{j,2} b_{j,2}=c_j$.
  
   We define 
   $a_{p,q}'=a_{i,p}+a_{j,q}+1_{\neq \{i,j\} }$ and $b_{p,q}'=b_{i,p}+b_{j,q}+c_{\neq \{i,j\} }$ for $p,q\in \{1,2\}$. We have 
   \begin{align*}
       a_{1,1}'y'= x' b_{1,1}'=a_{1,2}'y''= x' b_{1,2}'=0,\; a_{1,1}'b_{1,1}'=a_{1,2}'b_{1,2}'=c,\\
       a_{2,1}'y'=x'' b_{2,1}'=a_{2,2}'y''=x'' b_{2,2}'=0,\; a_{2,1}'b_{2,1}'=a_{2,2}'b_{2,2}'=c,
   \end{align*}
   whence \cref{c-factor-through-(x'_y')-etc}.
   
\textit{Case 3.2}.
     There exist $a_{j,1},a_{j,2},b_{j,1},b_{j,2},w',w''\in A_j$  such that $w=w'+w'$, $a_{j,1} y=a_{j,2}y=w'b_{j,1}=w''b_{j,2}=0$ and $a_{j,1} b_{j,1}=a_{j,2} b_{j,2}=c_j.$ By defining $a'=a_{i,1}+a_{j,1} +1_{\neq i,j },a''=a_{i,2}+a_{j,2} +1_{\neq i,j }, b'=b_{i,1}+b_{j,1} +c_{\neq i,j }$ and $b''=b_{i,2}+b_{j,2} +c_{\neq i,j }$, we have
      \begin{align*}
       a'y=a''y=x'b'=x'b''=0 \text{ and } a'b'=a''b''=c,
      \end{align*} 
whence \cref{c-gen-factor-through-(x_y)}.  

%\mk{Start here. Consider the case $n=2$.}

\textit{Case 4}. There exist $a_{i,1},a_{i,2},b_{i,1},b_{i,2},z',z''\in A_i$ with $z=z'+z''$, such that $a_{i,1}z'=a_{i,2}z''=xb_{i,1}=xb_{i,2}=0$ and  $a_{i,1} b_{i,1}=a_{i,2}b_{i,2}=c_i$.

Set $a_i=a_{i,1},b_i=b_{i,1}$.

 \textit{Case 4.1}.
     There exist $a_{j,1},a_{j,2},b_{j,1},b_{j,2}, y',y''\in A_j$  such that $y=y'+y''$, $a_{j,1}y'=a_{j,2}y''=w b_{j,1}=wb_{j,2}=0$ and $a_{j,1} b_{j,1}=a_{j,2} b_{j,2}=c_j$.
     
     By defining  $a'=a_i+a_{j,1} +1_{\neq i,j },a''=a_i+a_{j,2} +1_{\ne i,j }, b'=b_i+b_{j,1} +c_{\neq \{i,j\} }$ and $b''=b_i+b_{j,2} +c_{\neq i,j }$ we have
   \begin{align*}
       a'y'=a''y''=xb'=x b''=0  \mbox{ and }  a'b'=a''b''=c,
   \end{align*}
 whence \cref{c-gen-factor-through-(x_y)}. 
 
\textit{Case 4.2}.  There exist $a_{j,1},a_{j,2},b_{j,1},b_{j,2},w',w''\in A_j$  such that $w=w'+w'$, $a_{j,1} y=a_{j,2}y=w'b_{j,1}=w''b_{j,2}=0$ and $a_{j,1} b_{j,1}=a_{j,2} b_{j,2}=c_j.$
    
 By defining  $a'=a_i+a_{j,1} +1_{\neq i,j }, b'=b_i+b_{j,1} +c_{\neq \{i,j\} }$   we have 
  \begin{align*}
       a'y=xb'=0  \mbox{ and }  a'b'=c,
   \end{align*}
  whence \cref{c-gen-factor-through-(x_y)}.
\end{proof}

 \begin{lem}\label{V pp at c in A+B has pp at zero in A+B}
  Let $A_1,\dots, A_n$ be unital algebras, $0\ne c=(c_1,\dots, c_n)\in A_1 \oplus \dots \oplus A_n $ and for every $i$ let $S_i\subseteq A_i$ such that if $c_i\neq 0$ then $c_i$ has generalized factorization through zero products in $S_i$.
   
   Let $V:A\times A \to X$ be a bilinear map that has product property at $c.$ Then 
   \begin{enumerate}
       \item\label{pp at c in A+B has pp at zero in A+B 1} $V$ has product property at zero when restricted to each $S_i \times S_i,$
       \item\label{pp at c in A+B has pp at zero in A+B 2}If $n\geq 2$ and $i\neq j$, then $V(x,y)=0$ whenever $x\in S_i$ and $y\in S_j$ with $Ann_r(x)\cap S_i\neq \{ 0\}$ and $ Ann_l(y)\cap S_j\neq \{0\}$.
   \end{enumerate}
\end{lem}
\begin{proof}
\cref{pp at c in A+B has pp at zero in A+B 1} Fix $i$ and take $x,y\in S_i$ such that $xy=0.$ By \cref{lem gener fact direct sums} \cref{lem gener fact direct sums 1} the element $c$ has generalized factorization through $(x,y)$ and hence by \cref{lem gen zp factoriz implies pp at zero} we have $V(x,y)=0$.

\cref{pp at c in A+B has pp at zero in A+B 2} We apply \cref{lem gener fact direct sums}\cref{lem gener fact direct sums 2}. If $c$ has generalized factorization through $(x,y)$ then $V(x,y)=0$ by  \cref{lem gen zp factoriz implies pp at zero}. Otherwise, there exist $x',x'',y',y''$ such that $x=x'+x'',y=y'+y''$ and $V(x',y')=V(x',y'')= V(x'',y')=V(x'',y'')=0$ by \cref{lem fact thru zp xy=0 implies V(x_y)=0}, whence $V(x,y)=0$ by bilinearity.
 \end{proof}

\begin{lem}\label{lem sums of B(H) are zpd}
    Let $H_1,\dots,H_n$ be complex Hilbert spaces and $A=B(H_1)\oplus \dots \oplus B(H_n)$. Then $A$ is zero product determined (as an algebra).
\end{lem}
\begin{proof}
    In view of \cite[Theorem 2.3]{Bresar-zpd}, it is enough to prove that $A$ is generated by idempotents. In fact, each $B(H_k)$ is spanned by projections. Indeed, if $\dim(H_k)<\infty$ and $a\in B(H_k)$, then write $a=a_1-i a_2$, where $a_1,a_2$ are self-adjoint ($a_1=\frac 12(a+a^*)$ and $a_2=\frac i2(a-a^*)$). Now, by spectral theory, every self-adjoint element in $B(H_k)$ is a linear combination of projections. Hence, so is $a$. If $\dim(H_k)=\infty$, then use \cite[Corollary 2.3]{Pearcy-Topping-SmallSums}. 
    %Clearly, $A$ also is spanned by its projections and hence zero-product determined by \cite[Theorem 2.3]{Bresar-zpd}. 
\end{proof}

 \begin{defn}
     Let $A=\bigoplus_{\lb} B(H_\lb)$ be an atomic von Neumann algebra and $c\in A$. We say that $c$ satisfies the \textit{rank hypothesis}  if 
     \begin{enumerate}[label=(vNRH)]
         \item\label{VRH} $\rank(\pi_\lb(c) )\le\dim(H_\lb)-2$ for all $\lb$.
     \end{enumerate}
 \end{defn}

\begin{prop}\label{thm finite sums of B(H) determined at c}
Let $H_1,\dots,H_n$ be complex Hilbert spaces and $A=B(H_1)\oplus \dots \oplus B(H_n)$. Let $c=(c_1,\dots,c_n)\neq 0$ be an element satisfying $\cref{VRH}$ and such that each $\pi_{\lb}(c)$ is a finite-rank operator. Then
 $A$ is determined by products at $c$ (as an algebra). 
\end{prop}
\begin{proof}%$a=\frac{1}{2}(a+a^*)-i \frac{1}{2}(a-a^*)i$
Let $X$ be a vector space and  $V:A\times A\to X$ a  bilinear map  that has product property at $c$. Since $A$ is zero product determined by \cref{lem sums of B(H) are zpd}, it is enough to show that $V$ has product property at zero.

Let us first suppose that $\dim(H_i)<\infty$ for all $i$. In this case $A$ is a compact C$^*$-algebra, and the result is a particular case of \cref{thm main compact c*}. 

%In case all $H_i$ are infinite dimensional we have 
%\begin{enumerate}
%    \item Each $B(H_i)$ is spanned by its idempotent by \cite{Pearcy-Topping-SmallSums}[Corollary 2.3],
 %   \item $V$ has product property at zero when restricted to each $B(H_i) \times B(H_i)$ by \cref{lem gen zp factoriz implies pp at zero}. Now by (i) and \cite[Theorem 2.3]{Bresar-zpd} each $B(H_i)$ each $B(H_i)$ is determined by products at $c_i$, NOOO
 %   \item each $c_i$ has generalized factorizations through zero products in $A_i$
%\end{enumerate}

Otherwise, without loss of generality assume that $\dim(H_i)=\infty$, $1\le i\le s\le n$, and $\dim(H_i)<\infty$, $s+1\le i\le n$. Let $A_{fin}=B(H_{s+1})\oplus \dots \oplus B(H_n)$ and write $c_{fin}=c_{s+1}+\dots+ c_n$. Recall that by the assumptions on $c$ and \cref{lem compact star factor thu zp} if $c_{fin}\neq 0$ then $c_{fin}$ factorizes through zero products in $S=MinPI(A_{fin})$, since $A_{fin}$ is a compact C$^*$-algebra.

Take $x,y\in A$ such that $xy=0$. By bilinearity it is enough to show that $V(x,y)=0$ whenever $x\in B(H_i)\cup A_{fin}$ and $y\in B(H_j)\cup A_{fin}$ for some $1\leq i,j\leq s.$ We shall split the proof in several cases.

\textit{Case 1}. $x,y\in B(H_i)$ for some $i\leq s$. Recall from \cref{lem technical Julius} that  each $c_i\neq 0$ has generalized factorization through zero products in $S_i:=B(H_i)$. Moreover, if $c_{fin}\neq 0$ then $c_{fin}$ factorizes through zero products in $S.$ Thus, by 
\cref{V pp at c in A+B has pp at zero in A+B}\cref{pp at c in A+B has pp at zero in A+B 1} applied to $A_1\oplus \dots A_s \oplus A_{fin}$ and $c=c_1+\dots +c_s+c_{fin}$ we have $V(x,y)=0.$

\textit{Case 2}. $x,y\in A_{fin}$. The same argument given in Case 1 allows to show that $V(u,v)=0$ whenever $u,v\in S$ with $uv=0.$ Since $S$ is zp-dense in $A_{fin}$ by \cref{lem S zp-dense in compact C*} we have $V(x,y)=0.$

\textit{Case 3}. $x\in B(H_i)$ and $y\in A_{fin}$. Let $S_i:=P(B(H_i))\setminus \{1_{A_i}\}$ and $S:=MinPI(A_{fin})$. If $x\in S_i$ and $y\in S$, then apply  \cref{lem compact star factor thu zp,lem technical Julius} and \cref{V pp at c in A+B has pp at zero in A+B}\cref{pp at c in A+B has pp at zero in A+B 2} to show that 
$V(x,y)=0$. Since $\Span(S_i)=B(H_i)$ and $\Span(S)=A_{fin}$, we have $V(x,y)=0$ for all $x\in B(H_i),y\in A_{fin}$ by bilinearity.

\textit{Case 4}. $x\in A_{fin}$ and $y\in B(H_j)$. Analogous to Case 3.

\textit{Case 5}. $x\in B(H_i)$ and $y\in B(H_j).$ We apply \cref{lem technical Julius} and \cref{lem gener fact direct sums}\cref{lem gener fact direct sums 2} with $S_i=P(B(H_i))\setminus \{1_{A_i}\}$ and $S_j=P(B(H_j))\setminus \{1_{A_j}\}$ to show that 
$V(x,y)=0$ whenever $x\in S_i,y\in S_j$. Finally, since $\Span(S_i)=B(H_i)$ and $\Span(S_j)=B(H_j)$ we have $V(x,y)=0$ for all $x\in B(H_i),y\in B(H_j).$ 
\end{proof}

\begin{defn}
 Let $A$ be a C$^*$-algebra. We say that $c\in A$ has \textit{rank} $k\in \mathbb{N}\cup \{\infty\}$ and write $\rank(c)=k$, if either $c=0$ (whenever $k=0$) or $c=\sum_{i=1}^k \af_iu_i$, where $\{\af_i\}_{i=1}^k\subseteq\CC\setminus\{0\}$ and $\{u_i\}_{i=1}^k$ are mutually orthogonal non-zero minimal partial isometries (whenever $k>0$). If $\rank(c)<\infty$, then $c$ is said to be a \textit{finite-rank} element.
\end{defn}
If $A$ is a compact C$^*$-algebra  then the set of finite-rank elements coincides with $soc(A)$.

\begin{rem}\label{rem triple ideals}
    Let $A$ be a C$^*$-algebra. A norm closed subspace $I$ of $A$ is said to be a \textit{triple ideal} (respectively, \textit{Jordan ideal}) of $A$ if $\{A,A,I\}+\{A,I,A\}\subseteq I$ (respectively, $A\circ I\subseteq I$). By the comments in \cite[page 8]{Gar_Per_One_Par_JB} triple ideals in a C$^*$-algebra coincide with Jordan ideals. On the other hand, by \cite[Theorem 5.3]{Civin1965LieAJ} Jordan ideals coincide with (two-sided) ideals. Thus, in a C$^*$-algebra two-sided ideals coincide with triple ideals.
\end{rem}

 A C$^*$-algebra $A$ is said to be a \textit{von Neumann algebra} if it is dual to a Banach space (called \textit{predual} of $A$). It is known that each von Neumann algebra has a unique (up to isometry) predual \cite[Corollary 3.9]{Takesaki}. Let $A$ be a von Neumann algebra with non-zero minimal partial isometries. Then there exist two weak$^*$-closed ideals $A_{at}$ (the \textit{atomic part} of $A$) and $A_{d}$ (the \textit{diffuse} part of $A$) such that $MinPI(A)\sst A_{at}$, $A_{at}=\overline{\Span(MinPI(A))}^{w^*}$ and $A=A_{at}\oplus A_d$ (see \cite[Theorem 2]{FriedRusso_Predual}, \cref{rem triple ideals} and take into account that every von Neumann algebra is a JBW$^*$-triple). It is also known that $A_{at}\cong \bigoplus_{\lambda} B(H_{\lambda}) $ where each $H_{\lambda}$ is a complex Hilbert space (see \cite[Proposition 2]{FriedRusso_GelfNaim} and take into account that the only Cartan factors which are C$^*$-algebras are those of the form $B(H)$). The product in a von Neumann algebra is separately weak$^*$-continuous (see \cite[Theorem 1.7.8]{Sakai_Book}).

% \begin{lem}\label{lem pp isomorphism}
% Let $\phi:A\to B$ be an isomorphism between algebra. %Let us fix $c\in A$. Let $X$ be a vector space and %$V:A\times A\to X$ be a  bilinear map that has product %property at $c.$ Then $W:B\times B\to X$ defined by %$W(x,y)=V(\phi^{-1}(x), \phi^{-1}(x))$ has product %property at $\phi(c).$
 %\end{lem}

Whenever $A$ is a von Neumann algebra, by applying a $^*$-isomorphism, if necessary, we may always assume that $A_{at}=\bigoplus_{\lambda} B(H_{\lambda} )$. We also denote by $\pi_{\lambda}$ the projection of $A_{at}$ onto $B(H_{\lambda} ).$

 \begin{prop}\label{thm von neumann finite rank}
 Let $A$ be a von Neumann algebra with $A_{at}\ne\{0\}$ and $c\in A_{at}\setminus \{0\}$ a finite-rank element satisfying \cref{VRH}. If a bilinear map $V:A\times A\to X$ with values in a complex vector space $X$ has product property at $c$, then $V$ has product property at zero.
  \end{prop}
  \begin{proof}
  
  Since $\rank(c)<\infty$ and every  $u\in MinPI(A)$ lies in a unique $B(H_{\lambda})$, there exists a finite set  $I=\{\lambda_1,\ldots,\lambda_s \}$ such that $\pi_{\lambda}(c)\neq 0$ whenever $\lambda \in I$ and $\pi_{\lambda}(c)=0$ otherwise. Set $A_1=B(H_{\lambda_1})\oplus \ldots\oplus B(H_{\lambda_s})$. Then $A_1$ is a weak$^*$-closed ideal of $A$ and as a consequence of Proposition 1.10.5 in \cite{Sakai_Book} it follows that $A=A_1\oplus A_2$ where $A_2=A_1^{\perp}.$
  %Then $A_1,A_2$ are weak$^*$-closed ideals of $A$ (see for instance Lemma 2 in \cite{BurFerGarPer_Aut_Cont_2011}) and . 
  
  Let us take $x,y\in A$ such that $xy=0$.
  
 \textit{Case 1}. $x,y\in A_1$. By \cref{thm finite sums of B(H) determined at c} there exists a linear map $\varphi:A_1\to X$ such that $V(a,b)=\varphi(ab)$ for all $a,b\in A_1$. Thus, $V(x,y)=\vf(xy)=0$.
 
 \textit{Case 2}. $x,y\in A_2$. Then $c$ factorizes through $(x,y)$. Indeed, we have $1_{A_1}y=xc=0$ and $1_{A_1}c=c$. Hence $V(x,y)=0$ by \cref{lem fact thru zp xy=0 implies V(x_y)=0}. 
 
%  For each $i\in \{1,\dots,s\}$ we can apply \cref{V pp at c in A+B has pp at zero in A+B} to show that $V(x,y)=0$ whenever $x\in S_i,$

 \textit{Case 3}. $x\in A_1,y\in A_2$.
  Define 
  \begin{align*}
      S_i=
      \begin{cases}
        P(B(H_{\lambda_i}))\setminus \{1_{B(H_{\lambda_i})}\}, & \dim(H_{\lambda_i})=\infty,\\
        MinPI(B(H_{\lambda_i})), & \text{otherwise}.
      \end{cases}
  \end{align*}
  Let us suppose that $\pi_{\lambda_i}(x)\in S_i$ for all $i$. By  \cref{lem compact star factor thu zp} (if $\dim(H_{\lambda_i})<\infty$) or \cref{lem technical Julius} (if $\dim(H_{\lambda_i})=\infty$) we can assure that, for each $i$, $\pi_{\lambda_i}(c)$ has generalized factorization through zero products in $S_i$.
  
 Now observe that, for each $i$, there exists $z_i  \in S_i,z_i\neq 0$ such that $\pi_{\lambda_i}(x) z_i=0$. Indeed, if $\dim(H_{\lambda_i})<\infty$, this is contained in the first paragraph of the proof of \cref{lem S zp-dense in compact C*}. Otherwise, if $p\in S_i$, we can take $z_i=1_{B(H_{\lambda_i})}-p$. 

 Let us fix $i$ and $z_i$. We use generalized factorization through $(\pi_{\lambda_i}(x),z_i)$ to split the proof into two sub-cases. We shall  use the following short notations: $1_{\neq \lambda_i}:=1_{A_1}-1_{B(H_{\lambda_i})}$ and   $c_{\neq {\lambda_i}}:=c-\pi_{\lambda_i}(c)$.
  
  \textit{Case 3.1. } There exist $z_i',z_i'',a_1,a_2,b_1,b_2\in B(H_{\lambda_i})$ such that $z_i=z_i'+z_i''$, $a_1 z_i '=a_2 z_i''=\pi_{\lambda_i}(x)b_1=\pi_{\lambda_i}(x)b_2=0$ and $a_1b_1=a_2 b_2=\pi_{\lambda_i}(c)$.  

  Set $a=a_1+1_{\neq \lambda_i}$ and $b=b_1+c_{\neq \lambda_i}$.  We have
   \begin{align*}
       ay=\pi_{\lambda_i}(x) b=0  \mbox{ and }  ab=c,
   \end{align*}
 whence $c$ factorizes through $(\pi_{\lambda_i}(x),y)$, and consequently $V(\pi_{\lambda_i}(x),y)=0$ by \cref{lem fact thru zp xy=0 implies V(x_y)=0}.

 \textit{Case 3.2. } 
 There exist $x_i',x_i'',a_1,a_2,b_1,b_2\in B(H_{\lambda_i})$ such that $x_i'+x_i''=\pi_{\lambda_i}(x) ,$ $a_1z_i=a_2 z_i=x_i' b_1=x_i'' b_2=0$ and $a_1b_1=a_2 b_2=\pi_{\lambda_i}(c)$.  

  Set $a'=a_1+1_{\neq \lambda_i},a''=a_2+1_{\neq \lambda_i},b'=b_1+c_{\neq \lambda_i}$ and $b''=b_2+c_{\neq \lambda_i}.$ We have
   \begin{align*}
       a'y=a''y=x_i' b'=x_i'' b''=0  \mbox{ and }  a'b'=a''b''=c,
   \end{align*}
 thus $c$ has a generalized factorization through $(\pi_{\lambda_i}(x),y) $ and hence $V(\pi_{\lambda_i}(x),y)=0$ by \cref{lem gen zp factoriz implies pp at zero}.

 Thus, for all $i$, we have proved $V(\pi_{\lambda_i}(x),y)=0$ under the hypothesis that $\pi_{\lambda_i}(x)\in S_i$.
Now, $\Span(S_i)=B(H_{\lambda_i})$ for all $i$ by \cite[Corollary 2.3]{Pearcy-Topping-SmallSums} (if $\dim(H_{\lambda_i})=\infty$) and \cref{lem partial isometries}\cref{a=sum_i=1^k-lb_iu_i} (if $\dim(H_{\lambda_i})<\infty$). Therefore, by linearity, $V(x,y)=0$ for all $x\in B(H_{\lambda_i}),y\in A_2.$ Applying the linearity once again, we have $V(x,y)=\sum_{i=1}^s V(\pi_{\lambda_i}(x),y)=0$ whenever $x\in A_1,y\in A_2$.
  
   \textit{Case 4.} $x\in A_2,y\in A_1$. Similar to Case 3.
  \end{proof}
  
  As a consequence of our \cref{thm von neumann finite rank} together with Theorem 5.19 and Proposition 4.9 from \cite{Bresar-zpd}, we have the following.

\begin{thrm}\label{cor von neumann finit bounded}
 Let $A$ be a von Neumann algebra with $A_{at}\ne\{0\}$, $X$ a normed space and let $c\in A_{at}$ be a finite-rank element satisfying \cref{VRH}. Then $A$ is determined by products at $c$ (as a Banach algebra).
  \end{thrm}

A unital C$^*$-algebra is said to have \textit{real rank zero} if the set of self-adjoint invertible elements is dense in $A_{sa}$, where $A_{sa}=\{a\in A:a^*=a\}$. By \cite[Theorem 2.6]{Brown-Pederse-RRZero} a C$^*$-algebra $A$ has real rank zero if and only if the elements in $A_{sa}$ with finite spectra are dense in $A_{sa}.$ By spectral theory, if $a\in A_{sa}$ has finite spectra, then there exist mutually orthogonal projections $p_1,\dots,p_n\in A$ and $\lambda_1,\dots,\lambda_n\in \mathbb{R}$ such that $a=\sum_i \lambda_i p_i.$ Every von Neumann algebra has real rank zero (see  \cite[Proposition 1.3]{Brown-Pederse-RRZero}).

\begin{defn}
    Let $(X,\tau),(Y,\sigma)$ be topological spaces and $V:X\times X \to Y$. We say that $V$ is \textit{separately $\tau$-to-$\sigma$ continuous} if for every $x\in X$ the maps $V(-,x),V(x,-):X\to Y$ are continuous.

    In case $(Y,\sigma)=(X,\tau)$ we shall say that $V$ is \textit{separately $\tau$-continuous}.
\end{defn}

 The next result is probably well known. We include here a short proof inspired by Lemma 2.1 from \cite{Burgos-Sanchez-ZP-2013}.

 \begin{lem}\label{lem bilinear weak map zp}
 Let $A$ be a von Neumann algebra and $(X,\tau)$ a topological vector space. Let $V:A\times A \to X$ be a separately weak$^*$-to-$\tau$ continuous bilinear map that has product property at zero. Then
 $$ V(ab,c)=V(a,bc)$$ for all $a,b,c\in A.$
 \end{lem}
 \begin{proof}
  Let us fix  $p\in P(A).$ For $a,b\in A$ we have $ap(1-p)b=0$ and $a(1-p)pb=0.$ Thus we have 
  $$V(ap,b)=V(ap,pb)=V(a,pb) .$$ 
  Now let $p_1,\dots,p_n\in P(A),$ $\lambda_1,\dots,\lambda_n\in$  and set $x=\sum_{i=1}^n \lambda_i p_i$ then 
  \begin{equation}\label{eq lem tech normal zp form}
    V(ax,b)=\sum_{i=1}^n \lambda_i V(ap_i,b)=\sum_{i=1}^n \lambda_i V(a,p_ib)=V(a,xb).  
  \end{equation}

 Now, since $A$ has real rank zero, for every $x\in A_{sa}$, there exists a sequence $\{x_n\}\subseteq A_{sa}$ such that each $x_n$ is a real-linear combination of projections and  $x=\lim_n x_n,$ where the convergence of the latter sequence is in norm. Clearly, we also have $x=\lim_n x_n$ in the weak$^*$-topology. This fact together with separate weak$^*$-to-$\tau$ continuity of $V$ and \cref{eq lem tech normal zp form} give 
 \begin{align*}
     V(ax,b)=V(a,xb)
 \end{align*} for all $a,b\in A, x\in A_{sa}$. Finally, the statement follows by linearity together with $A=A_{sa}+iA_{sa}$.
  \end{proof}

\begin{rem}\label{rem with no type I finite summands}
 If $A$ is a von Neumann algebra with no summands of type $I_{fin}$ (that is, summands that are finite of type $I$), then $A$ is spanned by its projections (see \cite[Theorem V.1.19]{Takesaki} and \cite{Golds_Paszc_PAMS_1992} or the proof of Theorem 19 in \cite{BurFerGarPer_Aut_Cont_2011}), and continuity in \cref{lem bilinear weak map zp} is not needed.
 %\cite{Golds_Paszc_PAMS_1992}
\end{rem}

  Let $A$ be a von Neumann algebra. For each $a\in A$ there exists a projection $s_l(a)$ (resp., $s_r(a)$) called the \textit{left support} of $a$ (resp., the \textit{right support} of $a$) such that $Ann_l(a)=A(1-s_l(a))$ (resp., $Ann_r(a)=(1-s_r(a))A$). Moreover, $s_l(a)$ (resp., $s_r(a)$) is the least among all projections $p\in A$ such that $pa=a$ (resp., $ap=a$). If $a=a^*$ then $s_r(a)=s_l(a)$ is called the \textit{support projection} of $a$ and denoted by $s(a)$.
  If $a=u|a|$ is the polar decomposition of $a$, then $s_r(a)=u^*u=s(|a|)$ and $s_l(a)=uu^*=s(|a^*|)$ (see \cite[Definition 1.10.3 and Theorem 1.12.1]{Sakai_Book}).

\begin{lem}\label{lem zp implies cubes have zp}
   Let $A$ be a C$^*$-algebra and $a,b\in A$. If $aa^*abb^*b=0$ then $ab=0$.
\end{lem}
\begin{proof}

We shall see $a,b$ as elements in the von Neumann algebra $A^{**}$.
We multiply $aa^*abb^*b=0$ by $aa^*$ on the left and $b^*b$ on the right and apply induction to obtain
\begin{align*}
    (aa^*)^n ab(b^*b)^m=0
\end{align*} for all $n,m$. Consequently, for every  polynomial $p$ we have
\begin{align*}
    p(aa^*) ab(b^*b)^m=0.
\end{align*}

 Since $s(aa^*)$ lies in the weak$^*$-closure of $\{p(aa^*): p \mbox{ polynomial}\}$ (see \cite[Proposition 1.10.4]{Sakai_Book}), we have
\begin{align*}
    s(aa^*) ab(b^*b)^m=0.
\end{align*}
The same argument applied to $b^*b$ gives
\begin{align*}
    s(aa^*) abs(b^*b)=0.
\end{align*}

Finally, observe that $s(aa^*)a=a$ and $b s(b^*b)=b$.
\end{proof}

Let $A$ be a  C$^*$-algebra. It is known that for each $a\in A$ there exists a unique element $b\in A$ such that $bb^*b=a$. This element $b$ is called the \textit{odd cubic root} of $a$ and denoted by $a^{[\frac{1}{3}]}$ (see \cite{Kaup_Sing_Val}).

 \begin{thrm}\label{thhm von Neumann algebras}
 Let $A$ be a von Neumann algebra with $A_{at}\ne\{0\}$ and $(X,\tau)$ a topological vector space. Fix $c\in K(A)$ satisfying \cref{VRH}, and let $V:A\times A \to X$
 be a  separately weak*-to-$\tau$ continuous bilinear map that has product property at $c$. Then $V$ has product property at zero. Moreover, there exists a weak$^*$-to-$\tau$ continuous linear map $\vf:A\to X$ such that $V(a,b)=\vf(ab)$ for all $a,b\in A.$
 \end{thrm}
 \begin{proof}
  Since $soc(A)\ne\{0\}$, then $K(A)=\overline{soc(A)}^{\|.\|}$ is a non-zero two sided ideal of $A.$ Moreover, if $A=A_{at}\oplus A_d$ is the atomic decomposition of $A,$ then   $A_{at}=\overline{soc(A)}^{w^*}.$ We have $soc(A)\subseteq K(A) \subseteq A_{at}$. Since $A_{at}$ is weak$^*$-closed it follows that $\overline{K(A)}^{w^*}=A_{at}.$ 
  
   Let us take $x,y\in A$ such that $xy=0.$ 
   
   \textit{Case 1}. $x,y\in A_d$. By Cohen factorization property there exist $a,b\in A_{at}$ such that $ab=c$. Since $ay=bx=0$ and $ab=c$, we conclude that $V(x,y)=0$ by \cref{lem fact thru zp xy=0 implies V(x_y)=0}.
   
   \textit{Case 2}. $x,y\in A_{at}$. The proof will be divided in several cases from a restricted one to the general one.
   
   \textit{Case 2.1.} $x,y\in MinPI(A)$. By \cref{lem compact star factor thu zp} applied to $K(A)$ the element $c\in K(A)$ factorizes through $(x,y)$, and hence $V(x,y)=0$ by \cref{lem fact thru zp xy=0 implies V(x_y)=0}. 
   
   \textit{Case 2.2.} $x,y\in soc(A)$. Then by \cref{lem S zp-dense in compact C*} applied to $c$ as an element in $K(A)$ there exist $u_1,\dots,u_n,v_1,\dots, v_m\in MinPi(A)$ and $\alpha_1,\dots,\alpha_n,\beta_1,\dots, \beta_m \in \CC$ such that $u_i v_j=0,\forall i,j$, and $x=\sum_i \alpha_i u_i,y=\sum_j \beta_j v_j$. By bilinearity and the result of Case 2.1 we have 
   \begin{align}\label{eq 1 thm compact el in vn alg}
       V(x,y)=\sum_{i,j} \alpha_i \beta_j V(u_i,v_j)=0.
   \end{align}
   Thus, $V$ has product property at zero when restricted to $soc(A)$. 
   
   \textit{Case 2.3.} $x,y\in K(A)$. By \cref{lem S zp-dense in compact C*} there exist two sequences $\{x_n\},\{y_m\}\subseteq soc(A)$ such that $x_n y_m=0,\forall n,m$, and $x_n\to x,y_m\to y$ in the norm topology. Then $x_n\to x,y_m\to y$ also in the weak$^*$-topology. By \cref{eq 1 thm compact el in vn alg} we have 
   \begin{align}\label{eq 2 thm compact el in vn alg}
       V(x_n,y_m)=0, \forall n,m.
   \end{align}
 Now, taking into account that $V$ is separately weak*-to-$\tau$ continuous, we pass to the weak$^*$-limit on the left-hand side of \cref{eq 2 thm compact el in vn alg} to show that
   \begin{align*}%\label{eq V(x,ym)=0 thm compact el in vn alg}
       V(x,y_m)=0, \forall m.
   \end{align*}
A new application of separately weak*-to-$\tau$ continuity of $V$ yields $V(x,y)=0$. Thus, we have shown that $V$ has product property at zero when restricted to $K(A)$. 

\textit{Case 2.4.} Now consider the general case $x,y\in A_{at}$. By triple functional calculus, there exist odd cubic roots $x^{[\frac{1}{3}]},y^{[\frac{1}{3}]}\in A_{at}$. From $xy=0$ we have 
\begin{align*}
    x^{[\frac{1}{3}]}(x^{[\frac{1}{3}]})^*x^{[\frac{1}{3}]}y^{[\frac{1}{3}]}(y^{[\frac{1}{3}]})^*y^{[\frac{1}{3}]}=0,
\end{align*}
whence, by \cref{lem zp implies cubes have zp},
\begin{align}\label{eq dijsoint cubes}
    x^{[\frac{1}{3}]}y^{[\frac{1}{3}]}=0.
\end{align}
 By weak$^*$-density of $K(A)$ in $A_{at}$ there exist nets 
 $\{x'_{\lambda}\},\{y'_{\mu}\}\subseteq K(A)$ such that $x'_{\lambda}\to x^{[\frac{1}{3}]},y'_{\mu}\to y^{[\frac{1}{3}]}$ in the weak$^*$-topology. Define $x_{\lambda}=x'_{\lambda} (x^{[\frac{1}{3}]})^*x^{[\frac{1}{3}]}$ and $y_{\mu}=y^{[\frac{1}{3}]}(y^{[\frac{1}{3}]})^*y'_{\mu}$. Since $K(A)$ is an ideal, it follows that $\{x_{\lambda}\},\{y_{\mu}\}\subseteq K(A)$. Now, by  \cref{eq dijsoint cubes} we have $x_{\lambda} y_{\mu}=0,\forall \lambda,\mu$. Since $V$ has product property at zero when restricted to $K(A)$, by the result of Case 2.3 we have
 \begin{align}\label{eq xlambymu}
     V(x_{\lambda}, y_{\mu})=0,\forall \lambda,\mu.
 \end{align}
 By separate weak$^*$-continuity of the product in $A$ we have  $x_{\lambda}\to x,y_{\mu}\to y$ in the weak$^*$-topology. Now \cref{eq xlambymu} together with consecutive applications of 
 separate weak*-to-$\tau$ continuity of $V$ give $V(x,y)=0$.
 We have shown that $V$ has product property when restricted to $A_{at}$.

\textit{Case 3}. $x\in A_{at},y\in A_d.$ We first suppose that $x\in K(A).$ Let us fix $u\in MinPI(A)=MinPI(K(A))$ and let $v\in MinPI(A)=MinPI(K(A))$ such that $uv=0.$ By \cref{lem compact star factor thu zp} applied to $c$ as an elements in $K(A)$ we conclude that there exist $a,b\in K(A)$ such that $av=ub=0$ and $ab=c$. We have $ay=0,ub=0$ and $ab=c$ whence $V(u,y)=0$ for any $u\in MinPI(A)$ by \cref{lem fact thru zp xy=0 implies V(x_y)=0}. By linearity we have, $V(x,y)=0$ for every $x\in soc(A).$ Finally, by weak$^*$-density of $soc(A)$ in $A_{at}$ and separately weak*-to-$\tau$ continuity of $V$ we have $V(x,y)=0$.
  
\textit{Case 4}. $x\in A_{d},y\in A_{at}$. This case is analogous to Case 3.

Finally, for $x,y\in A$ the results follows by bilinearity and the previous cases.

Now define $\vf:A\to X$ by $\vf(a)=V(1,a).$ Then $\vf$ is weak$^*$-to-$\tau$ continuous and by \cref{lem bilinear weak map zp} we have $V(a,b)=V(1,ab)=\vf(ab)$ for all $a,b\in A.$
\end{proof}

\section{Applications}\label{sec-appl}

 The tools developed in the previous sections provide a unified approach to study several  homomorphism-like and derivation-like maps at a fixed point. Before showing some of these applications, we shall study pairs of maps that preserve zero products.  

\subsection{Pairs of linear maps that have product property}

Let $\vf,\psi:A \to B$ be a pair of linear mappings between two $K$-algebras and let us fix $c\in A.$ We say that the pair $(\vf,\psi)$ has \textit{product property at $c$} if the bilinear map $V(a,b)=\vf(a) \psi(b),a,b\in A$, has product property at $c$. In case $c=0$ we shall say that $(\vf,\psi)$ \textit{preserves zero products}. On the other hand, if $\vf=\psi$, then we say that $\vf$ \textit{has product property at $c$}, and, whenever $c=0$, $\vf$ \textit{preserves zero products}.

A similar concept  has been studied previously in \cite{Wolff-Disj-Pres,PerGar-Orth-Pairs}. Pairs of maps that preserve zero products on certain finite-dimensional central simple algebras were described by Bre\v{s}ar in \cite{Bresar_Mult_Algebra_2012}. More recently, Costara has characterized pairs of maps between matrix algebras that have product property at a fixed matrix~\cite{Costara-Racsam-2022}. Here we shall generalize some of Costara's results to infinite-dimensional C$^*$-algebras.

Recall that an algebra where every one-sided invertible element is invertible is called \emph{Dedekind-finite} or \emph{von Neumann finite}. A unital C$^*$-algebra is Dedekind-finite precisely when it is a finite C$^*$-algebra (see \cite[Lemma 5.1.2]{Rordam_Larsen_Laust_Book}).
 
 \begin{lem}\label{lem ab invertible imples a and b}
 Let $A$ be an associative algebra and $a,b\in A$. If $ab$ is invertible, then $a$ is right invertible and $b$ is left invertible.
 \end{lem}
 \begin{proof}
 We have $1=(ab)^{-1}ab=((ab)^{-1}a)b$ which shows that $b$ is left invertible. Similarly, $1=ab(ab)^{-1}=a(b (ab)^{-1}) $ shows that $a$ is right invertible.
 \end{proof}
 
\begin{lem}\label{vf(1)psi(ab)=vf(a)psi(b)=vf(ab)psi(1)}
    Let $A,B$ be Banach algebras with $A$ unital and zpd and $(\varphi,\psi)$ be a pair of bounded linear maps $A\to B$ that preserves zero products. Then
  \begin{align}\label{eq phi psi zp identity}
      \varphi(1)\psi(ab)=\varphi(a)\psi(b)=\varphi(ab)\psi(1), \forall a,b\in A.
  \end{align} 
  In particular,
  \begin{align}\label{vf(1)psi(a)=vf(a)psi(1)}
    \varphi(1)\psi(a)=\varphi(a) \psi(1), \forall a\in A.   
 \end{align}
\end{lem}
\begin{proof}
      By definition, the bilinear map $V:A\times A\to B,\; V(a,b)=\varphi(a) \psi(b)$, has product property at zero. By  \cite[Proposition 4.9]{Bresar-zpd} there exists a bounded linear map $T:A\to B$ such that $V(a,b)=T(ab)$ for all $a,b\in A.$ Then \cref{eq phi psi zp identity} follows.  Now we apply \cref{eq phi psi zp identity}  with $b=1$ to obtain \cref{vf(1)psi(a)=vf(a)psi(1)}.
\end{proof}

 \begin{lem}\label{lem zp pair are inverti preserving}
 Let $A,B$ be unital Banach algebras with $A$ zpd and $B$ Dedekind-finite and let $(\varphi,\psi)$ be a pair of bounded linear maps $A\to B$ that preserves zero products. Suppose further that both $\varphi(A)$ and $\psi(A)$ contain invertible elements. Then 
 \begin{enumerate}
     \item\label{vf-and-psi-preserve-inv} $\varphi$ and $\psi$ preserve invertibility;
     \item \label{vf-and-psi-same-ker} $\ker( \varphi)=\ker( \psi) $ is a two-sided ideal of $A$;
     \item \label{vf-and-psi-surjective}   $\varphi$  is surjective if and only if $\psi$ is surjective.
 \end{enumerate}
 \end{lem}
 
 \begin{proof}
  \cref{vf-and-psi-preserve-inv}. By assumption there exist $a,b\in A$ such that $\varphi(a)$ and $\psi(b)$ are in\-ver\-tible in $B$. Then $\varphi(a)\psi(b)$ is invertible. Hence, by \cref{lem ab invertible imples a and b,eq phi psi zp identity}, 
 %  By \cref{eq phi psi zp identity} we have
 %  \begin{align*}
 %   \varphi(1) \psi(ab)\psi(b)^{-1}\varphi(a)^{-1}=\varphi(a)\psi(b)\psi(b)^{-1}\varphi(a)^{-1}= 1   
 %  \end{align*}
 %    and
 %    \begin{align*}
 %      \psi(b)^{-1}\varphi(a)^{-1}\varphi(ab)\psi(1)=\psi(b)^{-1}\varphi(a)^{-1}\varphi(a)\psi(b)=1   
 %    \end{align*}
 % thus 
 $\varphi(1)$ and $\psi(1)$ are invertible because $B$ is Dedekind-finite. 
  
  Now take $a,b\in A$ such that $ab=1.$ A new application of  \cref{eq phi psi zp identity} gives
\begin{align*}
    \varphi(1)\psi(1)=\varphi(1)\psi(ab)=\varphi(a)\psi(b)
\end{align*}
  which shows that $\varphi(a) \psi(b)$ is invertible. By  \cref{lem ab invertible imples a and b} and the fact that $B$ is Dedekind-finite both $\varphi(a)$ and $\psi(b)$ are invertible. We have shown that $\vf$ and $\psi$ preserve invertibility.\smallskip
  
 \cref{vf-and-psi-same-ker}.
  Since $\varphi(1)$ and $\psi(1)$ are invertible, by \cref{vf(1)psi(a)=vf(a)psi(1)} we have $\varphi(a)=0 \Leftrightarrow \psi(a)=0$, that is, $\ker(\varphi)=\ker(\psi).$ Denote $I:=\ker(\varphi)=\ker(\psi)$. Take $a\in I.$ By  \cref{eq phi psi zp identity}, $\varphi(1)\psi(ab)=0$ for all $b\in A.$ Hence $aA \subseteq I.$  Symmetrically,  $Ab \subseteq I$ for any $b\in I$. Thus $I$ is a two-sided ideal of $A$.\smallskip
   
  \cref{vf-and-psi-surjective}.  The statement follows from \cref{vf(1)psi(a)=vf(a)psi(1)} and \cref{vf-and-psi-preserve-inv}.  
  \end{proof}

\begin{cor}\label{cor vf-bij-iff-psi-bij}
      Under the conditions of \cref{lem zp pair are inverti preserving} the map $\varphi$ is bijective if and only of $\psi$ is bijective. 
\end{cor}
\begin{rem}\label{rem lemma pairs w*}
   It is well known that every weak$^*$-continuous linear map is bounded. Thus items \cref{vf-and-psi-preserve-inv,vf-and-psi-same-ker,vf-and-psi-surjective} of \cref{lem zp pair are inverti preserving} are also true if $A$ and $B$ are von Neumann algebras with $B$ finite and $\vf$ and $\psi$ are weak$^*$-continuous, or if $A$ and $B$ are algebras with $A$ zpd and $B$ Dedekind-finite.
\end{rem}
%In case $A$ and $B$ are von Neumann algebras and $V$ is just separately weak$^*$-continuous we may apply \cref{lem zp pair are inverti preserving} to obtain \cref{eq phi psi zp identity}. If $A$ and $B$ are just algebras  \cref{eq phi psi zp identity} follows from the fact that $A$ is zpd.

  % The next result is folklore, we include here a short proof for completeness reasons.

 % \begin{lem}\label{lem techn T weak star continuous}
%Let $X,Y$ be two Banach spaces and $T:Y^*\to X^*$ be a weak$^*$-continuous bijective linear map. Then $T,T^{-1}$ are bounded and $T^{-1}$ is weak$^*$-continuous.   
%  \end{lem}
%\begin{proof}
%It is well known that $T$ is weak$^*$-continuous if and only there exists $S\in B(X,Y)$ such that $S^*=T$. Moreover, if $T$ is bijective then $S$ is bijective too.  By Banach isomorphism theorem $S^{-1}$ is also bounded. Then $T^{-1}=(S^*)^{-1}=(S^{-1})^*$ and hence $T^{-1}$ is weak$^*$-continuous. It is clear that both  $T$ and $T^{-1}$ are also bounded. 
%\end{proof}

  %Let $A$ be a unital complex Banach algebra and $a\in A$. The \textit{spectrum} of $a$ is the set $sp(a)=\{\lambda \in \CC: a-\lambda 1_A \mbox{ is not invertible} \}$. The spectrum is a non-empty compact set. 

 \begin{thrm}\label{prop zp pairs bij finite case} 
Let $A$ and $B$ be unital  C$^*$-algebras with $B$ finite  and let $(\varphi,\psi)$ be a pair of bounded  linear maps $A\to B$ that preserves zero products. Suppose further that both $\varphi(A)$ and $\psi(A)$ contain invertible elements. Then there exists a homomorphism $\rho:A\to B$ such that $\varphi=\varphi(1)\rho(-)$ and  $\psi= \rho(-) \psi(1).$ 
 \end{thrm}
 \begin{proof}
  By \cref{lem zp pair are inverti preserving}\cref{vf-and-psi-preserve-inv}, $\vf(1)$ and $\psi(1)$ are invertible elements. The mappings $\rho_1=\varphi(1)^{-1}\vf(-)$ and $\rho_2=\psi(-) \psi(1)^{-1}$ are bounded and unital. Moreover, the pair $(\rho_1,\rho_2)$ preserves zero products. Thus, we can apply \cref{vf(1)psi(a)=vf(a)psi(1)} to $(\rho_1,\rho_2)$ to show that $\rho_1=\rho_2.$ Set $\rho:=\rho_1=\rho_2$. The mapping $\rho$ is bounded, unital and preserves zero products. Then $\rho$ is a homomorphism by \cref{eq phi psi zp identity} applied to $(\rho,\rho)$. Clearly, $\vf=\varphi(1) \rho(-)$ and $\psi=\rho(-)\psi(1)$.
  \end{proof}

Let $\vf:A \to B$ be a linear map between $^*$-algebras. We say that $\vf$ is \textit{symmetric} if $\vf(a^*)=\vf(a)^*$ for all $a\in A$.  The next lemma is easy to prove.
\begin{lem}\label{lem invertible selfadjoint}
    Let $A$ be a unital $^*$-algebra. Then every one-sided invertible self-adjoint element is invertible. 
\end{lem}

  \begin{rem}\label{prop zp pairs bij symmetric case}
  \cref{prop zp pairs bij finite case} holds for non-finite $B$, whenever $\vf$ and $\psi$ are symmetric.
%Let $A,B$ be two unital C$^*$-algebras and let $(\varphi,\psi)$ be a pair of symmetric bounded linear maps $A\to B$ that preserves zero products. Suppose that $\varphi(A)$ and $\psi(A)$ contain an invertible element. Then there exists a  homomorphism $\rho:A\to B$ such that $\varphi=\varphi(1)\rho(-)$ and $\psi=\psi(1)\rho(-).$ 
 %\end{rem}
 %\begin{proof}

Indeed, since $\varphi$ and $\psi$ are symmetric, then $\varphi(1),\psi(1) \in B_{sa}.$ Furthermore, there exist $a,b$ such that $\vf(a)$ and $\psi(b)$ are invertible. Hence by \cref{vf(1)psi(ab)=vf(a)psi(b)=vf(ab)psi(1)} we see that $\varphi(1)$ is right invertible and $\psi(1)$ is left invertible.  Thus, $\varphi(1)$ and $\psi(1)$ are invertible thanks to \cref{lem invertible selfadjoint}. The rest of the proof is the same as that of \cref{prop zp pairs bij finite case}. 
%Define $\rho_1=\vf(1)^{-1}\vf(-)$ and $\rho_2=\psi(-)\psi(1)^{-1}$. Clearly, the pair $(\rho_1,\rho_2)$ preserves zero products and $\rho_1(1)=\rho_2(1)=1$ and hence $\rho_1=\rho_2$ thanks to \cref{vf(1)psi(a)=vf(a)psi(1)} applied to $(\rho_1,\rho_2)$. Now, taking into account that $\vf=\vf(1)\rho(-)$ and $\psi=\rho(-)\psi(1)$ \ref{eq phi psi zp identity} proves that $\rho$ is a homomorphism.
 \end{rem}

\begin{cor}\label{cor pairs with pp at c}
Let $A,B$ be von Neumann algebras such that $A_{at}\ne\{0\}$  and $B$ is finite. Let $\varphi,\psi:A\to B$ linear maps. Let us fix $c \in A_{at}$ satisfying \cref{VRH} and suppose that one of the following holds
\begin{enumerate}
    \item\label{cor pairs with pp at c item 1} $c\in soc(A)$ and $\varphi$, $\psi$ are bounded, or 
    \item\label{cor pairs with pp at c item 2} $c\in K(A)$ and $\varphi$, $\psi$ are weak$^*$-continuous.
    \end{enumerate}
    If $(\varphi,\psi)$ has product property at $c$ and  $\varphi(A)$ and $\psi(A)$ contain invertible elements then 
    there exists a homomorphism $ \rho :A\to B$ such that $\varphi=\varphi(1) \rho(-)$ and $\psi= \rho (-)\psi(1)$.
\end{cor}
\begin{proof}
 Define the bilinear map $V(a,b)=\varphi(a) \psi(b), a,b\in A$.  If \cref{cor pairs with pp at c item 1} holds then $V$ is bounded and we can apply \cref{thm von neumann finite rank} to show that $V$ has product property at zero. In case \cref{cor pairs with pp at c item 2}, $V$ is separately weak$^*$-continuous and  we apply \cref{thhm von Neumann algebras}. In any case, $V$ has product property at zero and thus the pair $(\varphi,\psi)$ preserves zero products. 
 By \cref{prop zp pairs bij finite case} or \cref{rem lemma pairs w*}  there exists a homomorphism $ \rho:A\to B$ such that $\varphi(a)=\varphi(1) \rho(-)$ and $\psi= \rho(-)\psi(1)$.
 \end{proof}

\begin{rem}
If in \cref{cor pairs with pp at c}  the maps $\psi$ and $\psi$ are symmetric then $B$ does not need to be finite. The proof of this fact can be obtained by replacing \cref{prop zp pairs bij symmetric case}  by \cref{prop zp pairs bij finite case} in the proof of \cref{cor pairs with pp at c}.
\end{rem}

%\begin{proof}
% Define the bilinear map $V(a,b)=\varphi(a) \psi(b)$. Since $\varphi$ and $\psi$ are weak$^*$-continuous and the pair  $(\varphi,\psi)$ has product property at $c$ then $V$ is separately weak$^*$-continuous and has product property at $c.$ By \cref{thhm von Neumann algebras} $V$ has product property at zero. Thus the pair $(\varphi,\psi)$ preserves zero products. Now an application of \cref{prop zp pairs bij symmetric case}
 %finishes the proof.
%\end{proof}

\subsection{Linear maps with product property at a fixed point}

When we deal with single maps instead of pairs of maps, we are able to obtain  more general results than in the case of pairs of maps.

Let $A$ be a C$^*$-algebra. The \textit{multiplier algebra} of $A$ is  defined as $M(A)=\{x\in A^{**}: ax,xa\in A, \forall a\in A\}.$ It is well known that $M(A)$ is a unital C$^*$-algebra and contains $A$ as an ideal. If $A$ is already unital then $M(A)=A$ (see \cite{Peds_Multipliers}).

\begin{thrm}\label{thm map-pp-at-c-compact case}
Let $A$ be a compact C$^*$-algebra and $B$ be a Banach algebra. Let $c\in A$ satisfting \cref{RH} and $\varphi:A\to B$ be a bounded linear map that has product property at $c$. Then $\varphi$ preserves zero products. Moreover, if $B$ is also a C$^*$-algebra and $\varphi$ is bijective,  then there exist $h\in Z(M(B))$ and an isomorphism $\rho: A\to B$ such that $\varphi=h \rho (-)$. 
\end{thrm}

\begin{proof}
 The bilinear map defined by $V(a,b)=\vf(a)\vf(b), a,b\in A$, is bounded and has product property at $c.$ By \cref{thm main compact c*} $\varphi$ preserves zero products. Moreover, if $\varphi$ is bijective and $B$ is a C$^*$-algebra, then $\varphi^{**}:A^{**}\to B^{**}$ is also bijective and by \cite[Proposition 2.3]{Gardella-Thiel_WeightedHom} $\varphi^{**}$ also preserves zero products. Now by \cite[Proposition 4.11]{Chebotar-Ke-Lee-Wong03} there exists $h\in Z(B^{**})$ and an isomorphism $\widetilde{\rho}:A^{**}\to B^{**}$ such that $\varphi^{**}(a)=h \widetilde{\rho}(a)$ for all $a\in A^{**}$. 
 
  For any $a,b\in A$ we have $h \varphi(ab)=h \varphi^{**}(ab)=h^2\widetilde{\rho}(a)\widetilde{\rho}(b)=\varphi^{**}(a)\varphi^{**}(b)=\varphi(a)\varphi(b)\in B$. By Cohen factorization property each $x\in A$ can be written as $x=ab$ for some $a,b\in A$. Hence, $h \varphi(x)\in B$ for all $x\in A.$ Since $\varphi$ is surjective and $h\in Z(B^{**})$, we have $hB,Bh\subseteq B$ and hence $h\in M(B).$ Therefore, $\widetilde{\rho}(A)=h^{-1}\varphi^{**}(A)=h^{-1} \varphi(A)=h^{-1}B\subseteq B$, so $\rho:=\widetilde{\rho}|_A$ is $B$-valued,  bijective and $\vf(a)=h \rho(a)$ for all $a\in A$. 
\end{proof}

Let $H$ be a complex Hilbert space. It is well known that $M(K(H))=B(H)$ and $Z(B(H))=\CC 1 $. If $A=K(H_1)$ and $B=K(H_2)$ in \cref{thm map-pp-at-c-compact case}, then by \cite[Corollary 3.2]{Chebotar-Ke-Lee-Wong03} injectivity is not needed in order to show that $\vf$ is a weighted homomorphism.

\begin{cor}
Let $H_i,i=1,2$, be complex Hilbert spaces and $c\in K(H_1)$ satisfying \cref{RH}. Let $\varphi: K(H_1)\to K(H_2)$ be a surjective bounded linear map that has product property at $c$. Then there exist $\lambda \in \CC$ and $S\in B(H_2,H_1)$ such that $\varphi(a)=\lambda S^{-1} a S$ for all $a\in K(H_1)$. 
\end{cor}

Let $A$ be a C$^*$-algebra and $e\in A$ a partial isometry. Then $e$ induces the Peirce decomposition $A=A_2(e)\oplus A_1(e)\oplus A_0(e)$ where $A_2(e)=ee^*Ae^*e,$ $A_1(e)=(1-ee^*)Ae^*e\oplus ee^*A(1-e^*e)$ and $A_0(e)=(1-ee^*)A(1-e^*e)$. The Peirce-2 space $A_2(e)$ is a unital C$^*$-algebra with unit element $e$ when endowed with the product $a \cdot_e b=ae^*b$ and involution $a^{*_e}=ea^*e$.
 Now let $A$ be a von Neumann algebra realized as a JBW$^*$-triple. For $a\in A$ there exists a partial isometry $r(a)\in A$, called the \emph{range partial isometry} of $a$, such that $a \in A_2(r(a))_{+}$ and $\{a\}^{\perp}=A_0(r(a))$ (see \cite[Lemma 3.3]{EdRut_compact_trip}, \cite[Lemma 1]{GarPer_op1} and \cite[Theorem 5.3]{EdRut_Orth_faces}). Moreover, $r(a)$ coincides with the partial isometry in the polar decomposition of $a$ (see \cite[Corollary 2.12]{GarPer_proj_less}). Taking into account that $r(a)=s(a)$ whenever $a\in A_{sa}$, \cite[Theorem 5.3]{EdRut_Orth_faces} together with induction give the following:

\begin{lem}\label{lem support a perp b}
    Let $A$ be a von Neumann algebra and $a_1,\dots,a_n\in A_{sa}$ with $a_i\perp a_j$ for all $i\neq j$. Then $s(a_1+\dots+a_n)=s(a_1)+\dots +s(a_n)$ and $s(a_i)\perp s(a_j), \forall i\neq j$. 
\end{lem}

Let $A$ be a C$^*$-algebra and $p,q\in P(A).$ We say that $q$ is less than or equal to $p$, denoted $q\leq p$, if $pq=q.$ It is easy to see that if $q\leq p$ then $A_2(q)\subseteq A_2(p)$.
%The following result is partially inspired by \cite[Theorem 2]{YeadonLP}.

\begin{lem}\label{lem tech disj pres von neumann algebras}
    Let $\vf:A\to B$ be a bounded linear map between two von Neumann algebras. If $\vf$ is symmetric and preserves zero products, then there exists a $^*$-homomorphism $\rho:A\to B$ such that
    $\vf=\vf(1)\rho(-)=\rho(-)\vf(1)$.
\end{lem}
\begin{proof}

Set $h=\vf(1)\in B_{sa}$. Set $s=s_{B^{**}}(h),$ where $ s_{B^{**}}(h)$ stands for the support projection of $h$ in $B^{**}$. Observe that $L_h,R_h:B_2^{**}(s)\to B_2^{**}(s), L_h(x)=hx,R_h(x)=xh$ are injective. Indeed, take $x\in B_2^{**}(s)$. Then $hx=hsx=h\cdot_{s} x=0$. Since $h\in B_2(s)_{sa}$ this implies that $x\perp s$. Hence $x=0$. The proof for $R_h$ is analogous.

Since $\vf$ is symmetric and preserves zero-products then $\vf$ preserves orthogonality, that is, $a\perp b$ implies $\vf(a)\perp \vf(b)$. Thus by \cite[Theorem 4.1]{OP_REV} there exists a Jordan $^*$-homomorphism $\rho:A\to B_2^{**}(s)$ such that $\vf(a)=\frac{1}{2}(h s\rho (a)+hs\rho(a))$ for all $a\in A.$ Moreover, $h\rho(a)=\rho(a)h^*=\rho(a)h$ holds for every $a\in A_{sa}$ (see the proof of \cite[Theorem 4.1]{OP_REV}). Thus $h \rho(-)=\rho(-)h$ since $A=A_{sa}+i A_{sa}$. Taking into account that $hs=sh=h$ we have $\vf(-)=h\rho(-)=\rho(-)h.$

 Observe that, since $s$ is a projection, then for $a,b\in B_2^{**}(s)$ we have $ab=a(s +(1-s))b=asb=a\cdot_{s} b$ and $a^*=(sas)^*=sa^*s=a^{*_{s}}.$ Thus  $\rho:A\to B_2^{**}(s)$ being  a Jordan $^*$-homomorphism implies that $\rho:A\to B^{**}$ also is a Jordan $^*$-homomorphism. Let us show that $\rho$ is a homomorphism.  Take $a,b\in A$ such that $ab=0$. We have $h^2 \rho(a)\rho(b)=\vf(a)\vf(b)=0$. Since $\rho(a)\rho(b)\in B_2^{**}(s)$, we have
   \begin{align*}
    0=h^2 \rho(a)\rho(b)=h^2 s\rho(a)\rho(b)=h^2\cdot_{s} \rho(a)\rho(b)=h\cdot_{s}(h\cdot_{s} \rho(a)\rho(b)),   
   \end{align*}
    whence $h\cdot_{s}\rho(a)\rho(b)=0$ which gives $\rho(a)\rho(b)=0$. Finally, $\rho$ is a homomorphism by \cite[Theorem 3.8]{Peralta2013ANO}.

To finish the proof, we shall show that $\rho$ is $B$-valued. Since $\Span(P(A))$ is norm dense in $A,$ it is enough to show that $\rho( P(A))\subseteq B.$ Let us fix $p\in P(A)$. Then  $\vf(p),\vf(1-p)\in B_{sa}$ and $p\perp (1-p)$. It follows that $\vf(p) \perp \vf(1-p)$. We define $\widetilde{\rho}(p):= s(\vf(p))$ (where the support projections are now taken inside $B$). Then
   \begin{align}
      \label{eq 0 lem tech disj pres von neumann algebras} \vf(p)=s(\vf(p)) \vf(p) s(\vf(p))=s(h) \vf(p) s(h) & \mbox{ and }\\
        \vf(p)=h \widetilde{\rho}(p)=\widetilde{\rho}(p)h, &\label{eq 1 lem tech disj pres von neumann algebras} 
   \end{align}
where in the above equalities we used the fact that $h=\vf(p)+\vf(1-p)$ and $s(\vf(p))\perp s(\vf(1-p))$ thanks to \cref{lem support a perp b}. 

   If follows from \cref{eq 0 lem tech disj pres von neumann algebras} that $\vf(P(A))\subseteq B_2(s(h))$. By orthogonality 
   \begin{align*}%\label{eq 2 lem tech disj pres von neumann algebras}
     s(\vf(p))=s(\vf(p)) s(\vf(p))s(\vf(p))=s(h)s(\vf(p))s(h),\ \forall p\in P(A).
   \end{align*}
   Thus $\widetilde{\rho}(P(A))\subseteq B_2(s(h))$. The mapping $\widetilde{\rho}: P(A)\to B_2(s(h)), p\mapsto \widetilde{\rho}(p)$, is well defined by \cref{eq 1 lem tech disj pres von neumann algebras} and the fact that left and right multiplications by $h$ restricted to $B_2(s(h))$ are injective.
   
Since $h$ is positive in $B_2(s(h))$, then it is also positive in $B_2^{**}(s(h))$. Now applying \cite[Lemma 4.3]{GarPer_proj_less} we see that $s=s_{B^{**}}(h)\leq s(h)$ (as projections in $B^{**}$) which implies that $B_2^{**}(s)\subseteq B_2^{**}(s(h)).$ As a consequence $\rho(A)\subseteq B_2^{**}(s(h)).$  Now if $p\in P(A)$ then $h\widetilde{\rho}(p)=\vf(a)=h \rho(p) $ whence  $\widetilde{\rho}(p)=\rho(p)$ since $\widetilde{\rho}(p),\rho(p)\in B_2^{**}(s(h)).$ Thus $\rho(P(A))\subseteq B_2(s(h))\subseteq B$ which finishes the proof.
%In particular, $\widetilde{\rho}$ is bounded and linear. 
%Observe that, since $A$ has real-rank zero, $\Span(P(A))$ is dense in $A$.
%, we see that $\widetilde{\rho}$ has a unique extension to a bounded linear map $\rho:A\to B$. It is clear that $\widetilde{\rho}(A)\subseteq B_2(s(h))\subseteq B_2^{**}(s(h)).$ 
%Take $a\in A,$ there exists $(a_n)\subseteq \Span(P(A))$ such that $a_n\to a$ in norm. Then $(\rho(a_n))=(\widetilde{\rho}(a_n))\subseteq B.$ We have
%\begin{align*}
 %  \rho(a)= \lim_n \rho(a_n)\in B
%\end{align*} 
\end{proof}

%\mk{Start here. Think about a definition of product property for bounded bilinear maps on Banach algebras with values in normed spaces (include the continuity of $\vf$).}

\begin{thrm}\label{thm map with pp at finite rank element}
Let $A,B$ be von Neumann algebras such that $A_{at}\ne\{0\}$.  Let $c\in soc(A)$ satisfying \cref{VRH} and $\varphi :A\to B$ be a linear map that has product property at $c.$ Then $\varphi$ preserves zero products. Moreover, if $\varphi$ is bounded and symmetric (resp. bounded and surjective), then there exist $h\in \vf(A)'$ and a $^*$-homomorphism (resp. surjective homomorphism) $\rho:A\to B$ such that $\varphi=h \rho(-).$ 
% \begin{enumerate}
%     \item \label{item 1 thm map with pp at finite rank element} $\varphi$ is bounded and surjective,
%     \item  \label{item 2 thm map with pp at finite rank element} $\varphi$ is bounded and symmetric,
% \end{enumerate}

\end{thrm}
\begin{proof}
 By \cref{thm von neumann finite rank} applied to the bilinear map $V(a,b)=\vf(a) \vf(b)$ we show that $\varphi$ preserves zero products.
 
 If $\varphi$ is bounded and symmetric, then apply \cref{lem tech disj pres von neumann algebras}. If $\varphi$ is bounded and surjective, then the result is a direct consequence of \cite[Theorem 4.1 (vi)]{Chebotar-Ke-Lee-Wong03}.
\end{proof}

\begin{rem}
  \cref{thm map with pp at finite rank element} also holds if we assume that $\vf$ is weak$^*$-continuous and $c\in K(A)$. The proof is the same as that of \cref{thm map with pp at finite rank element} but in the first step we apply  \cref{thhm von Neumann algebras} instead of \cref{thm von neumann finite rank}.
\end{rem}

\subsection{Derivations at a fixed point}

Let $A$ be algebra and $X$ an $A$-bimodule. We say that a linear map $\delta:A\to X$ is a \textit{derivation} if  for all $a,b\in A$:
 \begin{align*}%\label{def derivation}
     \delta(ab)=\delta(a)b+a\delta(b).
 \end{align*}
Given a fixed $c\in A,$ a linear map $\delta:A\to X$ is said to be a \textit{derivation at $c$} (or \textit{derivable at} $c$) if for all $a,b\in A$:
\begin{align*}%\label{ab=z=>V(ab)=V(1z)=V(z1)}
    ab=c \impl \delta(ab)=\delta(a)b+a\delta(b).
\end{align*}

Let $A$ be a C$^*$-algebra and $X$ a Banach A-bimodule. We say that $X$ is \textit{essential} if $\overline{\Span(\{axb: a,b\in A,x\in X\})}^{\|.\|}=X.$ Recall that $X^{**}$ is a Banach $A^{**}$-bimodule and also a Banach $A$-bimodule (see \cite[Theorem 2.6.15 (iii)]{HG_Dales_Book} for more details).

The proof of the following result is essentially contained in the proof of \cite[Theorem 6]{AbAdJamPeralta-AntiDer-2020}.
%It can also be proved by replacing  \cite[Lemma 2.3]{Fad_Gah_Deriv_at_ORTH} by \cite[Lemma 5]{AbAdJamPeralta-AntiDer-2020} in the proof of \cite[Theorem 3.1 (i)]{Fad_Gah_Deriv_at_ORTH}.

\begin{prop}\label{prop derivable at zero}
     Let $\delta:A\to X$ be a bounded linear map, where $A$ is a C$^*$-algebra and $X$ is an essential Banach A-bimodule. The following are equivalent:
     \begin{enumerate}
         \item $\delta$ is a derivation at zero.
         \item There exists a derivation $d:A\to X^{**}$ and an element $\xi \in X^{**}$ such that $a \xi=\xi a$ and $\delta(a)=d(a)+\xi a$ for all $a\in A$.
     \end{enumerate}
     Moreover, if $A$ is unital or $X$ is a dual Banach space then $\xi \in X.$
\end{prop}

\begin{cor}\label{cor derivation at c}
Let $A$ be C$^*$-algebra with $soc(A)\neq \{0 \}$ and $X$ a Banach $A$-bimodule. Fix $c\in A$ and let $\delta:A\to X$ be a derivation at $c$. Assume that one of the following holds: 
\begin{enumerate}
    \item \label{cor derivation at c 1} $A$ and $X$ are dual Banach spaces, $c\in soc(A)$ and $\delta$ is bounded,
    \item \label{cor derivation at c 2}  $A$ and $X$ are dual Banach spaces, $c\in K(A)$ and $\delta$ is  weak$^*$-continuous,
    \item \label{cor derivation at c 3}  $A$ is a compact C$^*$-algebra and $\delta$ is bounded.
\end{enumerate}
If $c$ satisfies \cref{VRH} (resp., \cref{RH} if \cref{cor derivation at c 3} holds) then $\delta$ is a derivation at zero. Moreover, there exist a derivation $d:A\to X$ and $\xi \in X$ (resp. $d:A\to X^{**}$ and $\xi \in X^{**}$ if \cref{cor derivation at c 3} holds) such that $\xi c=0$,  $\delta(a)=d(a)+\xi a$ and $\xi a= a\xi$ for all $a\in A$.
\end{cor}
\begin{proof}
 Define $V:A\times A\to X$ by $V(a,b)=\delta(ab)-\delta(a)b-a\delta(b).$
 Since $\delta$ is a derivation at $c$ it follows that $V$ has product property at $c$.
 
 \cref{cor derivation at c 1} In this case $V$ is bounded and has product property at zero by \cref{thm von neumann finite rank}, hence $\delta$ is a (bounded) derivation at zero.

 \cref{cor derivation at c 2} In this case $V$ is separately weak$^*$-continuous and has product property at zero by \cref{thhm von Neumann algebras}, hence $\delta$ is a (bounded) derivation at zero.

 \cref{cor derivation at c 3} In this case  $\delta$ is a (bounded) derivation at zero by \cref{thm main compact c*}.

 By \cref{prop derivable at zero} there exists a derivation $d:A\to X^{**}$ and an element $\xi \in X^{**}$ such that $a \xi=\xi a$ and $\delta(a)=d(a)+\xi a$ for all $a\in A$. Moreover, take $a,b\in A$ such that $ab=c$. We have
 \begin{align*}
    a(d(b)+\xi b)+(d(a)+\xi a)b&=a\delta(b)+\delta(a)b=\delta(ab)\\
    &=d(ab)+\xi ab=d(a)b+ad(b)+\xi ab,
 \end{align*}
 whence $\xi c=\xi ab=0$.

 In cases \cref{cor derivation at c 1,cor derivation at c 2} we have $d:A\to X$ and $\xi \in X$ since $A$ is unital.
\end{proof}
	\section*{Acknowledgements}
	Jorge J. Garcés was partially supported by grant PID2021-122126NB-C31 funded by MCIN/AEI/10.13039/501100011033 and Junta de Andalucía grant FQM375. Mykola Khrypchenko was partially supported by CMUP, member of LASI, which is financed by national funds through FCT --- Fundação para a Ciência e a Tecnologia, I.P., under the project with reference UIDB/00144/2020. The authors are grateful to the referee for a careful reading of the paper and the suggested useful improvements.
\bibliography{bibl}{}
	\bibliographystyle{acm}

\end{document}